\documentclass[12pt,a4paper]{article}
\RequirePackage{amsmath,amsthm}
\RequirePackage[colorlinks,citecolor=blue,urlcolor=blue,pagebackref=true]{hyperref}
\RequirePackage[T1]{fontenc}
\usepackage[british,french]{babel}
\usepackage{amssymb,math,locenglish}
\usepackage{fourier}
\renewcommand{\cut}[1]{}

\usepackage[usenames]{color}
\usepackage{tikz}

\newcommand{\LastUpdate}{\eurtoday}

\begin{document}
%\begin{frontmatter}

\begin{center}
\baselineskip=5mm
\textbf
{\Large
{%\textsf
{Explosion, implosion, and moments of passage times\\ for continuous-time Markov chains:\\
\vskip3mm
a semimartingale approach}}
 }
 \footnote{
%{\em 1990 Physics and Astronomy Classification Scheme:} \\
}
\vskip1cm
\parbox[t]{14cm}{
Mikhail {\sc  Menshikov}$^{a}$ and Dimitri {\sc Petritis}$^b$\\
\vskip5mm   
{\scriptsize  
\baselineskip=5mm
a. Department of Mathematical Sciences,
University of Durham,
South Road,
Durham DH1 3LE,United Kingdom,
Mikhail.Menshikov@durham.ac.uk\\
%\vskip5mm
b. Institut de Recherche
Math\'ematique, Universit\'e de Rennes I,
Campus de Beaulieu, 
35042 Rennes, France, Dimitri.Petritis@univ-rennes1.fr}}
\vskip10mm
\LastUpdate
\end{center}
\vskip5mm
%\titlerunning{Continuous-time Markov chains}

\selectlanguage{british}

\begin{abstract}  We establish general theorems quantifying the notion of recurrence --- through an estimation of the moments of passage times --- for irreducible continuous-time Markov chains on countably infinite state spaces.
{\cut{These theorems are applied to models whose jump processes have a deeply critical behaviour.}} Sharp conditions of occurrence of the phenomenon of explosion are also obtained.
A new phenomenon of implosion is introduced and sharp conditions for its
occurrence are proven. 
The general results  are illustrated by treating models having a difficult behaviour even in discrete time.
\vskip5mm
\selectlanguage{french}
\centerline{\textbf{Résumé}}
\vskip2mm
Nous établissons des théorèmes généraux qui quantifient la notion de récurrence --- à travers l'estimation des moments de temps de passage ---
pour des chaînes de Markov à temps continu sur des espaces d'états dénombrablement  infinis. Des conditions fines garantissant l'apparition du phénomène d'explosion sont obtenues. Un nouveau phénomène d'implosion est introduit et des conditions fines pour son apparition sont aussi démontrées. Les résultats généraux sont illustrés sur des modèles ayant un comportement non-trivial même lorsqu'ils sont considérés en temps discret.

\end{abstract}
\selectlanguage{british}

\section{Introduction, notation, and main results}
\label{sec:intro}
\subsection{Introduction}

\label{ssec:intro}

In this paper
we establish general theorems quantifying the notion of recurrence
--- by studying which moments of the passage time exist --- for 
irreducible continuous-time Markov chains $\xi=(\xi_t)_{t\in[0,\infty[}$ on a countable space $\BbX$ in critical regimes. 
\cut{The aforementioned criticality means that even the discrete time chain --- sampled at the moments of jumps  --- has already a non-trivial behaviour. }

Models of discrete-time Markov chains with non-trivial  behaviour include reflected random walks in wedges of dimension $d= 2$ \cite{AspIasMen,FayMalMen,MenPet-wedge}, Lamperti processes \cite{Lam1,Lam2}, etc.
These chains exhibit strange polynomial behaviour. In the null recurrent case some (but not all) moments of the
random time needed to reach a finite set are obtained by transforming the 
discrete-time Markov chain into a discrete-time semimartingale via its mapping
through a Lyapunov function \cite{FayMalMen}.

There exist in the literature powerful theorems \cite{AspIasMen}, applicable to discrete-time critical Markov chains, allowing to determine which moments of the passage time exist. Beyond their theoretical interest, such results can be used to estimate the decay of the stationary measure \cite{MenPop}, and even the speed of convergence towards the stationary measure.
The first aim of this paper is to show that theorems concerning moments of passage times can be usefully and instrumentally extended to the continuous time situation. 

Continuous-time Markov chains have an additional feature compared to discrete-time ones, namely, on each visited state they spend a random holding time (exponentially distributed)
defined as the difference between successive jump times. We consider the space inhomogeneous situation where the parameters $\gamma_x\in\BbR_+$ of the exponential holding times   (the inverse of their expectation) are unbounded, i.e.\ 
$\sup_{x\in\BbX} \gamma_x=+\infty$. In such situations, the phenomenon of \textit{explosion} can occur for transient chains. Chung \cite{Chu} has established  that
the condition $\sum_{n=1}^\infty 1/{\gamma_{\tilde{\xi}_n}}<+\infty$, where $\tilde{\xi}_n$ is the position of the chain immediately after the $n$-th jump has
occurred,    
is equivalent to explosion. However this condition is very difficult to check since is global i.e.\ 
requires the knowledge of the entire trajectory of the embedded Markov chain. Later, sufficient conditions for explosion --- whose validity can be verified by local estimates ---   have been introduced.  Sufficiently sharp conditions of explosion and non-explosion, applicable only to Markov chains on countably infinite subsets of non-negative reals, are given in \cite{HamKle,KerKle}. In \cite{Str}, a sufficient condition of explosion is established for Markov chains on general countable sets; similar sufficient conditions of explosion are established in
\cite{Wag} for Markov chains on locally compact separable metric spaces.

The second aim of this paper is to show that the phenomenon of explosion can also be  sharply studied
by the use of Lyapunov functions and to establish \textit{locally verifiable}
 conditions for explosion/non explosion for Markov chains on arbitrary graphs. 
This method is applied to models that even without explosion are difficult to study. More fundamentally, the development of the  semimartingale method has been largely inspired with these specific critical models in mind
(such as the cascade of $k$-critically transient Lamperti models or of reflected random walks on quarter planes) that seem refractory to known methods. 

Finally, we demonstrate a new phenomenon, we termed \textit{implosion}
(see definition \ref{def:implo} below), reminiscent of the Döblin's condition for general Markov chains \cite{Dob}, occurring in the case $\sup_{x\in\BbX} \gamma_x=\infty$. We show that this phenomenon can also be explored with the help of Lyapunov functions.

\subsection{Notation} 
\label{ssec:notat}
Throughout this paper, $\BbX$ denotes the state space of our Markov chains;
it denotes an abstract denumerably infinite set, equipped with its full $\sigma$-algebra $\cX=\cP(\BbX)$. It is worth stressing here that, generally, this space is not naturally partially ordered. The graph whose edges are the ones induced by the stochastic matrix,  when equipped with the natural graph metric on $\BbX$ need not be isometrically embedable  into $\BbZ^d$ for some $d$. Since the definition of a continuous-time Markov chain on a denumerable set is standard (see \cite{Chu}, for instance), we introduce below its usual
equivalent description in terms of holding times and embedded Markov chain merely for the purpose of  establishing our notation.

Denote by $\Gamma=(\Gamma_{xy})_{x,y\in\BbX}$ the \textit{generator} of the continuous Markov chain, namely the 
matrix satisfying:
$\Gamma_{xy} \geq 0$ if $y\ne x$ and $\Gamma_{xx}=-\gamma_x$, where $\gamma_x=\sum_{y\in\BbX\setminus\{x\}} \Gamma_{xy}$. We assume that for all $x\in\BbX$, we have $\gamma_x<\infty$.

We construct a stochastic Markovian matrix $P=(P_{xy})_{x,y\in\BbX}$ out of $\Gamma$ by defining 
\[P_{xy}=\left\{\begin{array}{ll}
\frac{\Gamma_{xy}}{\gamma_x} & \textrm{if }\ \gamma_x\ne 0\\
0& \textrm{if }\ \gamma_x=0, 
\end{array}\right. \textrm{for} \ y\ne x, \ \ \textrm{and}\ \
P_{xx} =\left\{\begin{array}{ll}
0 & \textrm{if }\ \gamma_x\ne 0\\
1 & \textrm{if }\ \gamma_x= 0.
\end{array}\right.\]
The kernel $P$ defines a discrete-time $(\BbX,P)$-Markov chain 
$\tilde{\xi}=(\tilde{\xi}_n)_{n\in\BbN}$ termed the
\textit{Markov chain embedded at the moments of jumps}. 
To avoid irrelevant complications, we always assume that this Markov chain
is \textit{irreducible}. 

Define a sequence $\sigma=(\sigma_n)_{n\geq 1}$ of \textit{random holding times} distributed,  conditionally on $\tilde{\xi}$, according to an exponential law. More precisely, consider
\[\BbP(\sigma_{n}\in ds|\tilde{\xi})=\gamma_{\tilde{\xi}_{n-1}}
\exp(- s \gamma_{\tilde{\xi}_{n-1}}) \id_{\BbR_+}(s) ds, \ n\geq 1,\]
so that $\BbE(\sigma_{n}|\tilde{\xi})=1/{\gamma_{\tilde{\xi}_{n-1}}}$.
The sequence $J=(J_n)_{n\in\BbN}$ of \textit{random jump times} is defined accordingly by $J_0=0$ and for $n\geq 1$ by $J_n=\sum_{k=1}^n \sigma_k$.
The life time is denoted $\zeta=\lim_{n\rightarrow\infty} J_n$ and we say
that the (not yet defined continuous-time) Markov chain \textit{explodes} on  $\{\zeta<\infty\}$, while it does not
explode (or is regular, or conservative) on $\{\zeta=\infty\}$.

\begin{rema}{}
The parameter  $\gamma_x$ must be interpreted as the  proper frequency of the internal clock  of the Markov chain multiplicatively modulating the local speed of the chain. We always assume that for all $x\in\BbX$, $0<\gamma_x<\infty$. The case  
$0<\underline{\gamma}:=\inf_{x\in\BbX} \gamma_x\leq \sup_{x\in\BbX} =:\overline{\gamma}<\infty$ is elementary because the chain can be stochastically controlled by two Markov chains whose internal clocks tick respectively at constant pace $\underline{\gamma}$ and $\overline{\gamma}$. Therefore the sole interesting cases are 
\begin{itemize}
\item $\sup_x \gamma_x=\infty$: the internal clock ticks unboundedly fast (leading to an unbounded local speed of the chain),
\item $\inf_x \gamma_x=0$: the internal clock ticks arbitrarily slowly (leading to a local speed that can be arbitrarily close to 0).
\end{itemize}
\end{rema}

To have a unified description of both explosive and non-explosive processes, we can extend the state space into $\hat{\BbX}=\BbX\cup\{\partial\}$ by adjoining a special absorbing state $\partial$. The continuous-time Markov chain is then the
c\`adl\`ag process $\xi=(\xi_t)_{t\in[0,\infty[}$
defined by 
\[
\xi_0=\tilde{\xi}_0 \ \ \textrm{and}\ \ \
\xi_t=\left\{\begin{array}{ll}
\sum_{n\in\BbN} \tilde{\xi}_n\id_{[J_n,J_{n+1}[}(t) & \textrm{ for }\ 
0<t<\zeta\\
\partial & \textrm{ for }\ 
t\geq \zeta.
\end{array}\right.
\]
Note that although $\BbX$ is merely a set (i.e.\ no internal composition rule is defined on it), the above ``sum'' is well-defined since for every  fixed $t$ only one term survives. 
We refer the reader to standard texts (for instance \cite{Chu,Nor}) for the proof of the equivalence between $\xi$ and $(\tilde{\xi}, J)$. The natural right continuous filtration $(\cF_t)_{t\in[0,+\infty[}$ is defined as usual through
$\cF_t=\sigma(\xi_s: s\leq t)$; similarly  $\cF_{t-}=\sigma(\xi_s: s<t)$, and  
$\tilde{\cF}_n=\sigma(\tilde{\xi}_k, k\leq n)$ for $n\in\BbN$.
%%%% Klebaner p. 250
For an arbitrary $(\cF_t)$-stopping time $\tau$, we denote as usual its past $\sigma$-algebra $\cF_\tau=\{A\in\cF_\infty: A\cap\{\tau\leq t\}\in\cF_t\}$ and its strict past $\sigma$-algebra $\cF_{\tau-}=\sigma \{A\cap \{t<\tau\}; t\geq 0, A\in\cF_t\}\vee \cF_0$.
Since it is immediate to show that $\tau$ is $\cF_{\tau-}$-measurable, 
we conclude that
the only information contained in $\cF_{J_{n+1}}$ but not in $\cF_{J_{n+1}-}$
is contained within the random variable $\tilde{\xi}_{n+1}$, i.e.\ the position where the chain jumps at the moment $J_{n+1}$.

If $A\in\cX$, we denote by $\tau_A=\inf\{t\geq 0: \xi_t\in A\}$ the $(\cF_t)$-stopping time of reaching $A$.  

A dual notion to explosion is that of \textit{implosion}:
\begin{defi}{}
\label{def:implo}
Let $(\xi_t)_{t\in[0,\infty[}$ be a continuous-time Markov chain on $\BbX$ and let $A\subset\BbX$ be a proper subset of $\BbX$. We say that the Markov chain \textit{implodes} towards $A$ if $\exists K>0: \forall x\in A^c,
\BbE_x(\tau_A)\leq K.$
\end{defi}

\begin{rema}{} 
It will be shown in proposition \ref{pro:implosion-irreducibility} that in the case the set $A$ is finite and the chain is irreducible, implosion towards $A$ means implosion towards any state. In this situation, we speak about \textit{implosion of the chain}.
\end{rema}
It is worth noticing that some other definitions of implosion can be introduced; all convey the same idea of reaching a finite set from an arbitrary initial point within a random time that can be uniformly (in the initial point) bounded in some appropriate stochastic sense. We stick at the form introduced in the previous definition because it is easier to 
establish necessary and sufficient conditions for its occurrence and is easier to illustrate on specific problems (see \S \ref{sec:exam}).

We use the notational conventions of \cite{Revuz} to denote measurable functions, namely
$m\cX =\{f:\BbX\rightarrow \BbR| f\ \textrm{is $\cX$-measurable}\}$ with all possible decorations: $b\cX$ to denote bounded measurable functions,
$m\cX_{+}$ to denote non-negative measurable functions, etc.
For $f\in m\cX_+$ and $\alpha>0$, we denote by $\sublev{f}{\alpha}$ the \textit{sublevel set} of $f$ of height $\alpha$ defined by
\[\sublev{f}{\alpha}:=\{x\in\BbX: f(x) \leq \alpha\}.\]
We recall that a function $f\in m\cX_+$ is \textit{unbounded} if $\sup_{x\in\BbX} f(x)=
+\infty$ while  \textit{tends to infinity} ($f\rightarrow\infty$) when for every $n\in\BbN$
the sublevel set $\sublev{f}{n}$ is finite. 
Measurable functions $f$ defined on $\BbX$ can be extended to functions $\hat{f}$, defined on $\hat{\BbX}$,
by $\hat{f}(x)=f(x)$ for all $x\in\BbX$ and $\hat{f}(\partial):=0$ (with obvious extension of the $\sigma$-algebra).

We denote by $\Dom(\Gamma)=\{f\in m\cX: \sum_{y\in\BbX\setminus\{x\}} \Gamma_{xy}
|f(y)|<+\infty, \forall x\in\BbX\}$ the domain of the generator and by  $\Dom_+(\Gamma)$ the set of  non-negative functions in the domain.
The action of the generator $\Gamma$ on $f\in\Dom(\Gamma)$ reads then:
$\Gamma f(x):=\sum_{y\in\BbX} \Gamma_{xy} f(y)$.

\subsection{Main results}
\label{ssec:main}

We recall once more that in the whole paper  we make the following

\begin{glas}{} The chain embedded at the moments of jumps is  \textit{irreducible} and $0<\gamma_x<\infty$ for all $x\in\BbX$.
\end{glas}

We are now in position to state our main results concerning the use of Lyapunov function to obtain, through semimartingale theorems, precise and locally verifiable conditions on the parameters of the chain allowing us to establish existence or non-existence of moments of passage times, explosion or implosion phenomena. The proofs of these results are given in section \ref{sec:proof}; the section \ref{sec:exam} treats some critical models (especially \ref{ssec:lamperti} and \ref{ssec:quadrant}) that are difficult  to study even in discrete time, illustrating thus both the power of our methods and giving specific examples on how to use them. 

\subsubsection{Existence or non-existence of moments of passage times}

\begin{theo} %%%%%%%%%%%%%%%%%%%%%%%%%%%%
\label{thm:moments} 
Let $f\in \Dom_+(\Gamma)$ be such that $f\rightarrow \infty$. 
\begin{enumerate}
\item If there exist constants $a>0$, $c>0$  and $p>0$ such that
$f^p\in\Dom_+(\Gamma)$ and
\[\Gamma f^p(x)\leq  -cf^{p-2}(x), \forall x\not\in \sublev{f}{a}, \]
then $\BbE_{x}(\tau_{\sublev{f}{a}}^q)<+\infty$ for all $q< p/2$ and all $x\in\BbX$. 
\item 
Let $g \in m\cX_+$.
If there exist
\begin{enumerate}
\item  a constant $b>0$ such that $f\leq  bg$,
\item constants $a>0$ and $c_1>0$ such that
$\Gamma g(x) \geq -c_1$ for $x\not\in \sublev{g}{a}$, 
\item constants $c_2>0$ and $r>1$ such that $g^r\in\Dom(\Gamma)$ and $\Gamma g^r (x) \leq c_2 g^{r-1}(x)$ for $x\not\in \sublev{g}{a}$,
 and 
\item a constant $p>0$ such that $f^p\in\Dom(\Gamma)$ and $\Gamma f^p\geq 0$ for $x\not\in \sublev{f}{ab}$,
\end{enumerate}
 then $\BbE_{x}(\tau_{\sublev{f}{a}}^q)=+\infty$ for all $q>p$ and all
$x\not\in \sublev{f}{a}$.
\end{enumerate}
\end{theo}

In many cases,  the function $g$, whose existence is assumed in statement 2, of the above theorem can be chosen as  $g=f$ (with obviously $b=1$). In such situations  we have to check $\Gamma f^r \leq c_2 f^{r-1}$ for some $r>1$  and find a $p>0$ such that $\Gamma f^p \geq 0$ on the appropriate sets.  However, in the case of the problem studied in \S \ref{ssec:quadrant}, for instance, the full-fledged version of the previous theorem is needed.

Note that the conditions $f\in\Dom_+(\Gamma)$ and $f^p\in\Dom_+(\Gamma)$
for some $p>0$ holding simultaneously imply that $f^q\in\Dom_+(\Gamma)$ for all $q$ in the interval with end points $1$ and $p$. When  $\tau_A$ is integrable, the chain is positive recurrent. In the null recurrent situation however, $\tau_A$ is almost surely finite
but not integrable; nevertheless, some fractional moments $\BbE(\tau_A^q)$ with $q<1$ can exist. Similarly, in the positive recurrent case, some higher moments
 $\BbE(\tau_A^q)$ with $q>1$ may fail to exist.

 When $p=2$, the first statement in the above theorem \ref{thm:moments} simplifies to the following: if $\Gamma f(x)\leq -\epsilon$, for some $\epsilon>0$ and for $x$ outside a finite set $F$, then the passage time $\BbE_x(\tau_F^q)<\infty$ for all $x\in\BbX$ and all $q<1$. 
As a matter of fact, in this situation, we have a stronger result, expressed in the form of the following
 \begin{theo}
 \label{thm:moments-stronger}
 The following are equivalent:
 \begin{enumerate}
 \item The chain is positive recurrent.
 \item 
 There exist a pair  $(F,f)$, where $F$ is a finite proper subset of\ \ $\BbX$\ \ and $f$ a function  in $\Dom_+(\Gamma)$  verifying\ \ \ $\inf_{x\in F^c} f(x) > \max_{x\in F} f(x)$, and a constant $\epsilon>0$   such that  
$\Gamma f(x) \leq -\epsilon$ for all $x\not\in F$.
\end{enumerate}
 \end{theo}
Obviously, positive recurrence implies \textit{a fortiori} that 
$\BbE_x(\tau_F)<\infty$.

\begin{rema}{}
It is obvious that the pair $(F,f)$ in the above theorem \ref{thm:moments-stronger} is not uniquely determined. Mostly, it will be possible to chose a function $f\rightarrow \infty$ and $F$ as the 
sublevel set of $f$ at height $a$, for some $a>0$. Sometimes it will be possible to chose  the function $f$ uniformly bounded; this case will be further considered in theorem \ref{thm:implo} and leads to implosion.  It is also immediate that if $f$ verifies the conditions $\Gamma f(x) \leq -\epsilon$ for $x\in F^c$ and  $\inf_{x\in F^c} f(x) > \max_{x\in F} f(x)$, then the function $g= f\id_{F^c}$ verifies \textit{a fortiori} the same conditions. 
\end{rema}

 If only establishing occurrence of recurrence or transience is sought, the first generalisation of Foster's criteria to the continuous-time case was given in the unpublished technical report \cite{Miller-Foster}. Notice however that the method in that paper is subjected to the same important restriction as in the original paper of Foster \cite{Fos}, namely the semi-martingale condition must be verified everywhere but in one point.
 
 If $\gamma_x$ is bounded away from $0$ and $\infty$, then since the Markov chain can be stochastically controlled by two Markov chains with constant  $\gamma_x$ reading respectively $\gamma_x=\underline{\gamma}$ and $\gamma_x=\overline{\gamma}$ for all $x$, the previous result is the straightforward generalisation of the theorems 1 and 2 of \cite{AspIasMen} 
 established in the case of discrete time; as a matter of fact, in the case of constant $\gamma_x$,  the complete behaviour of the continuous time process is encoded solely into the jump chain and since results in 
 \cite{AspIasMen} were optimal, the present theorem introduces no improvement.  Only the  interesting cases of $\sup_{x\in\BbX}\gamma_x=\infty$ or $\inf_{x\in\BbX}\gamma_x=0$
 are studied in the sequel; the models studied in \S \ref{sec:exam},
 illustrate how the theorem can be used in critical cases to obtain locally verifiable 
 conditions of existence/non-existence of moments of reaching times.
 The process $X_t=f(\xi_t)$, the image of the Markov chain through the Lyapunov function $f$, can be shown to be a semimartingale; therefore, the semimartingale approach will prove instrumental as was the case in discrete time chains.
 
 \subsubsection{Explosion}
 The next results concern explosion obtained again using Lyapunov functions. It is worth noting that although explosion can only occur in the transient case,
 the next result is strongly reminiscent of the Foster's criterion \cite{Fos}
 for positive recurrence!  
 \begin{theo} %%%%%%%%%%%%%%%%%%%%%%%%%%%%%%
 \label{thm:explo}
 The following are equivalent:
 \begin{enumerate}
 \item There exist $f\in\Dom_+(\Gamma)$ strictly positive and 
 $\epsilon >0$ such that  $\Gamma f(x)\leq -\epsilon$
 for all $x\in\BbX$.
 \item The explosion time $\zeta$ satisfies $\BbE_x \zeta <+\infty$ for all $x\in\BbX$.
 \end{enumerate}
 \end{theo}
 
 \begin{rema}{}
 Comparison of statements 1 of theorem \ref{thm:moments} and 1 of theorem \ref{thm:explo} demands some comments. The conditions of theorem \ref{thm:moments} imply that $\sublev{f}{a}$ is a finite set and \textit{necessarily not empty}.  For $p=2$ and $F=f^p$, the condition reads $\Gamma F(x) \leq -\epsilon$ outside some finite set and this implies recurrence. In theorem \ref{thm:explo} the condtion $\Gamma f\leq -\epsilon$ must hold \textit{everywhere} and this implies transience.
\end{rema}
The one-side implication 
$[\Gamma f(x)\leq -\epsilon, \forall x]
\Rightarrow [\BbP_x(\zeta<+\infty)=1, \forall x]$ is already established, for $f\geq0$,   in the second part of
theorem 4.3.6 of  \cite{Str}. Here, modulo the (seemingly) slightly more stringent requirement $f>0$,  we strengthen the result from almost sure finiteness to integrability  and prove equivalence instead of mere implication.  

\begin{prop}
\label{pro:explo}
Let $f\in\Dom_+(\Gamma)$ be a strictly positive bounded function and denote $b=\sup_{x\in\BbX} f(x)$;  assume there exists an increasing --- not necessarily strictly --- function $g:\BbR_+\setminus\{0\}\rightarrow\BbR_+\setminus\{0\}$ such that its inverse has an integrable singularity at 0, i.e.\  
$\int_0^b\frac{1}{g(y)}dy<\infty$.
If  we have
$\Gamma f(x)\leq -g(f(x))$ for all $x\in \BbX$, then $\BbE_x\zeta <\infty$ for all $x$.
\end{prop}

The previous proposition, although stating the conditions on $\Gamma f$  quite differently than in theorem \ref{thm:explo}, will be shown to follow from the former. This proposition is interesting only when $\inf_{x\in\BbX}g\circ f(x)=0$  because then  the condition required in \ref{pro:explo} is weaker than the uniform requirement $\Gamma f(x)\leq -\epsilon$ for all $x$ of theorem \ref{thm:explo}.

If for some $x\in\BbX$, explosion (i.e.\ $\BbP_x(\zeta<+\infty)>0$) occurs, irreducibility of the chain implies that the process remains explosive for all starting points $x\in\BbX$. However, 
since the phenomenon of explosion can only occur in the transient case,
examples (see \S \ref{ssec:fingers}) can be constructed with non-trivial tail boundary so that for some initial $x\in\BbX$, we have both 
$0<\BbP_x(\zeta<+\infty)<1$. Additionally, the previous theorems established conditions that guarantee $\BbE_x\zeta <\infty$ (implying explosion). However examples are constructed where $\BbP_x(\zeta=\infty)=0$ (explosion does not occur) while $\BbE_x\zeta=\infty$. 
It is therefore important to have results on conditional explosion.

\begin{theo} %%%%%%%%%%%%%%%%%%%%%%%%%%%%%
\label{thm:cond-explo}
Let  $A$ be a proper (finite or infinite) subset of $\BbX$ and $f\in \Dom_+(\Gamma)$ such that 
\begin{itemize}
\item there exists $x_0\not\in A$ with $f(x_0)<\inf_{x\in A} f(x)$,
\item $\Gamma f(x) \leq -\epsilon$ on $A^c$
\end{itemize}
Then $\BbE_x(\zeta| \tau_A =\infty) <\infty$.
\end{theo}

The previous results (theorems \ref{thm:explo} and \ref{thm:cond-explo})  --- through unconditional or conditional integrability of the explosion time $\zeta$ ---  give conditions establishing explosion. 
For the theorem \ref{thm:explo} these conditions are even necessary and sufficient. It is nevertheless extremely difficult in general to prove that a function satisfying the conditions of the theorems does not exist. 
We need therefore a more manageable criterion guaranteeing non-explosion. Such a result is provided by the following

\begin{theo}  %%%%%%%%%%%%%%%%%%%%%%%%%%%
\label{thm:non-explo}
Let $f\in\Dom_+(\Gamma)$.  If
\begin{itemize}
\item $f\rightarrow \infty$,
\item there exists an increasing (not necessarily strictly) function 
$g:\BbR_+\rightarrow\BbR_+$ whose inverse is locally integrable but has non integrable tail (i.e.\ 
$G(z):=\int_{0} ^z \frac{dy}{g(y)}<+\infty$ for all $z\in\BbR_+$ but $\lim_{z\rightarrow \infty} G(z)=\infty$), and
\item $\Gamma f(x) \leq g(f(x))$ for all $x\in \BbX$,
\end{itemize}
 then 
$\BbP_x(\zeta=+\infty)=1$ for all $x\in\BbX$.
\end{theo}

The idea of the proof of the theorem \ref{thm:non-explo} relies on the intuitive idea that if $f(x)\leq g(f(x))$ for all $x$, then $\BbE(f(X_t))$ cannot grow very fast with time and since $f\rightarrow \infty$
the process itself cannot grow fast either. The same idea has been used in \cite{EibWag} to prove non-explosion for Markov chains on metric separable spaces.
Our result relies on the powerful  ideas developed in  the proof of  theorems 1 and 2 of \cite{KerKle} and of theorem 4.1 of \cite{EibWag} but  improves the original results in several respects. 
In first place, our result is valid on arbitrary denumerably infinite state spaces $\BbX$ (not necessarily subsets of $\BbR$); in particular, it can cope with models on higher dimensional lattices (like random walks in $\BbZ^d$ or reflected random walks in quadrants).
Additionally, even for processes on denumerably infinite subsets of $\BbR$, our result covers critical regimes  such as those exhibited by the Lamperti model (see \S \ref{ssec:lamperti}),  a ``crash test'' model, recalcitrant to the methods of  \cite{KerKle}.

\subsubsection{Implosion}

Finally, we state results about implosion. 
\begin{theo} %%%%%%%%%%%%%%%%%%%%%%%%%%%%
\label{thm:implo}
Suppose the embedded chain is recurrent.
\begin{enumerate}
\item
The following are equivalent:
\begin{itemize}
\item There exist a function $f\in\Dom_+(\Gamma)$ such that $\sup_{x\in\BbX} f(x)=b<\infty$, an   $a\in ]0,b[$, such that $\sublev{f}{a}$ is finite, and an $\epsilon>0$  such that
\[x\not\in \sublev{f}{a} \Rightarrow \Gamma f(x)\leq -\epsilon.\]
\item For every finite $A\in\cX$, there exists a constant $C:=C_A>0$ such that 
\[x\not\in A \Rightarrow \BbE_x\tau_{A} \leq C,\]
 (hence there is implosion towards $A$ and subsequently the chain is implosive).
\end{itemize}
\item Let $f\in\Dom_+(\Gamma)$ be such that  $f\rightarrow \infty$ and assume there exist constants $a>0$, $c>0$, $\epsilon>0$, and $r>1$ such that  $f^r\in\Dom_+(\Gamma)$. If further
\begin{itemize}
\item $\Gamma f(x) \geq -\epsilon$, for all $x\not\in \sublev{f}{a}$, and
\item $\Gamma f^r(x) \leq cf^{r-1}(x)$, for all $x\not\in  \sublev{f}{a}$, 
\end{itemize}
then the chain does not implode towards $ \sublev{f}{a}$.
\end{enumerate}
\end{theo}

We conclude this section by the following
\begin{prop}
\label{pro:implosion}
Suppose the embedded chain is recurrent. 
Let $f\in\Dom_+(\Gamma)$ be strictly positive and such that $\sup_{x\in\BbX} f(x)=b<\infty$; assume further that for any $a$ such that $0<a<b$, the sublevel set $\sublev{f}{a}$ is finite.
Let $g:[0,b]\rightarrow \BbR_+$ be an increasing function such that
$B:=\int_0^b \frac{dy}{g(y)}<\infty$. If $\Gamma f(x) \leq -g(f(x))$ for all $x\not\in  \sublev{f}{a}$ then $\BbE_x \tau_{ \sublev{f}{a}}\leq B$ for all $x\not\in  \sublev{f}{a}$, i.e.\ the chain implodes towards $\sublev{f}{a}$.
\end{prop}

In some applications, it is quite difficult to guess  immediately the form of the function $f$ satisfying the uniform condition $\Gamma f(x) \leq -\epsilon$ required for the first statement of the previous theorem \ref{thm:implo} to apply. It is sometimes more convenient to check merely that $\Gamma f(x) \leq -g(f(x))$ for some function $g$ vanishing at $0$ in some controlled way. 
The proposition \ref{pro:implosion} --- although does not improve the already optimal statement 1 of the theorem \ref{thm:implo} --- provides us with a convenient alternative condition to be checked.

\section{Proof of the main theorems}
\label{sec:proof}

%%%%%%%%%%%%%%%%%%%%%%%%%%%%%%%%%%
 We have already introduced the notion of $\Dom(\Gamma)$. A related notion is that of  locally $p$-integrable functions,  
defined as $\ell^p(\Gamma)=\{f\in m\cX: \sum_{y\in\BbX} \Gamma_{xy}
|f(y)-f(x)|^p<+\infty, \forall x\in\BbX\}$, for some $p>0$. Obviously $\ell^1(\Gamma)=\Dom(\Gamma)$.  In accordance to our notational convention on decorations, $\ell_+^p(\Gamma)$ will denote positive $p$-integrable functions. For $f\in \ell^1(\Gamma)$, we define the \textit{local $f$-drift}
of the embedded Markov chain as the random variable
\[\Delta^f_{n+1}:=\Delta f(\tilde{\xi}_{n+1}):=f(\tilde{\xi}_{n+1})-f(\tilde{\xi}_{n+1-})=f(\tilde{\xi}_{n+1})-f(\tilde{\xi}_{n}),\] the \textit{local mean $f$-drift} by 
\[m_f(x):=\BbE(\Delta^f_{n+1}|
\tilde{\xi}_{n}=x)=\sum_{y\in\BbX} P_{xy} (f(y)-f(x))=\BbE_x\Delta^f_{1},\] and for  $\rho\geq1$ and 
$f\in \ell^\rho(\Gamma)$ the \textit{$\rho$-moment of the local $f$-drift}
 by
\[ v_f^{(\rho)}(x):=\BbE\large(|\Delta^f_{n+1}|^\rho|
\tilde{\xi}_{n}=x\large)=\sum_{y\in\BbX} P_{xy} |f(y)-f(x)|^\rho=\BbE_x|\Delta^f_{1}|^\rho.\] We always write shortly $v_f(x):=v_f^{(2)}(x)$.
The action of the generator $\Gamma$ on $f$ reads 
\[\Gamma f(x):=\sum_{y\in\BbX} \Gamma_{xy} f(y)=\gamma_x m_f(x).\]
Since $(\xi_t)_t$ is a pure jump process, the  process  $(X_t)_t$ transformed by $f\in\Dom(\Gamma)$, i.e.\ $X_t=f(\xi_t)$, is also a pure jump process reading, for $t<\zeta$,
\[X_t=f(\xi_t)=\sum_{k=0} ^\infty f(\tilde{\xi}_k)\id_{[J_k,J_{k+1}[}(t)=X_0+
\sum_{k=0} ^\infty \Delta^f_{k+1}\id_{]0,t]} (J_k).\]
If there is no explosion, the process $(X_t)$ is a $(\cF_t)$-semimartingale admitting the decomposition $X_t=X_0+M_t+A_t$, where $M_t$ is a martingale vanishing at $0$ and the predictable compensator reads \cite{Kle}
\[A_t=\int_{]0,t]} \Gamma f(\xi_{s-}) ds= \int_{[0,t[}  \Gamma f(\xi_{s}) ds.\]
Notice that, although not explicitly marked, 
$(X_t)$, $(M_t)$,  and $(A_t)$ depend on $f$. Therefore, for any admissible $f$ we have $dX_t=dM_t+dA_t=dM_t + \Gamma f(\xi_{t-})dt$.

\subsection{Proof of the theorems \ref{thm:moments} and \ref{thm:moments-stronger} on moments of passage times}
\label{ssec:semi-mart}
We start by the theorem \ref{thm:moments-stronger} that --- although stronger  --- is much simpler to prove.

\begin{lemm}
\label{lem:expec-time}
Let $(\Omega, \cG, (\cG_t), \BbP)$ be a filtered probability space and 
$(Y_t)$ a $(\cG_t)$-adapted process taking values in $[0, \infty[$.
Let $c\geq 0$ and denote $T=\inf\{t\geq 0: 
Y_t \leq c\}$; suppose that there exists $\epsilon >0$ such that 
$\BbE(dY_t |\cG_{t-})\leq -\epsilon dt$ on the event $\{T\geq t\}$. Suppose that $Y_0=y$ almost surely, for some $y\in[c,\infty[$.
Then $\BbE_y(T)\leq\frac{y}{\epsilon}$.
\end{lemm}

\begin{proof}
Since $\{T\geq t\}\in\cG_{t-}$, the hypothesis of the lemma reads
$\BbE(dY_{t\wedge T}|\cG_{t-})\leq -\epsilon \id_{\{T\geq t\} }dt$. Taking expectations and integrating over time,  yields:
$0\leq \BbE(Y_{t\wedge T})\leq y -\epsilon \int_0^t \BbP(T\geq s) ds$
which implies $\epsilon\BbE(T)\leq  y$.
\end{proof}

\noindent\textit{Proof of the theorem \ref{thm:moments-stronger}.}
\begin{description}
\item{$[2\Rightarrow 1]:$}
Let $X_t=f(\xi_t)$; then the condition 2 reads 
$\BbE(dX_{t\wedge \tau_F}|\cG_{t-})\leq -\epsilon \id_{\{\tau_F\geq t\} }dt$. Hence,  in accordance with the previous lemma \ref{lem:expec-time}, we get 
$\BbE_x(\tau_F)\leq \frac{f(x)}{\epsilon}$ for every $x\not\in F$.
Now, let $x\in F$. Then
\begin{eqnarray*}
\BbE_x (\tau_F)&=& \BbE_x (\tau_F|\tilde{\xi}_1\in F)\BbP_x(\tilde{\xi}_1\in F)+ \BbE_x (\tau_F|\tilde{\xi}_1\not\in F)\BbP_x(\tilde{\xi}_1\not\in F)\\
&\leq & \sup_{x\in F} \left(\frac{1}{\gamma_x}+\sum_{y\in\BbX} P_{xy} \frac{f(y)}{\epsilon}\right)<\infty,
\end{eqnarray*}
the finiteness of the last expression being guaranteed by the conditions $f\in\Dom_+(\Gamma)$ and the finiteness of the set $F$.
Positive recurrence follows then from irreducibility of the chain.
\item{$[1\Rightarrow2]:$}
Let $F=\{z\}$ for some fixed $z\in\BbX$; positiive recurrence of the chain implies that $\BbE_x(\tau_F)<\infty$ for all $x\in\BbX$.
Fix some $\epsilon>0$ and define 
\[f(x)= \left\{\begin{array}{ll} \BbE_x(\tau_F) & \textrm{if }\ x\not\in F,\\
0 & \textrm{otherwise. }
\end{array}\right.\]
Then, $m_f(x) \leq \sum_{y\not= x}P_{xy} \epsilon\BbE_y(\tau_F) -\epsilon\BbE_x( \tau_F)=
\epsilon\BbE_x(\tau_F-\sigma_1)-\epsilon \BbE_x(\tau_F)=-\epsilon\BbE_x(\sigma_1)= -\frac{\epsilon}{\gamma_x}$, for all $x\not\in F$. It follows that 
$\Gamma f(x) =\gamma_x m_f(x)\leq -\epsilon$ outside $F$. 
\end{description}
Writing $X_t=f(\xi_t)$, we check immediately that $(X_t)$ satisfies the requirements of the lemma \ref{lem:expec-time} above.
\eproof

The proof of the theorem \ref{thm:moments} is quite technical and will be split into several steps formulated as independent lemmata and propositions on semimartingales that may have an interest \textit{per se}. As a matter of fact, we use these intermediate results to prove various results of very different nature.
\begin{lemm}
\label{lem:pgetwo}
Let $f\in\Dom_+(\Gamma)$   tending to infinity, $p\geq 2$,  and 
 $a>0$. Use the abbreviation $A:=\sublev{f}{a}$.  Denote $X_t=f(\xi_t)$ and assume further
that $f^p\in\Dom_+(\Gamma)$. 
If there exists $c>0$ such that
\[\Gamma f^p(x)\leq -cf^{p-2}(x), \forall x\not\in A, \]
then the process defined by 
$Z_t=(X^2_{\tau_A\wedge t} +\frac{c}{p/2} \tau_A\wedge t) ^{p/2}$
 is a non-negative supermartingale.
\end{lemm}

\begin{proof}
Introducing the predictable decomposition $1=\id_{\{\tau_{A}<t\}}+
\id_{\{\tau_A\ge t\}}$, we get 
$\BbE(dZ_t |\cF_{t-})= \BbE(d(X^2_{t} +\frac{c_1}{p/2}  t) ^{p/2}|\cF_{t-})
\id_{\{\tau_A\ge t\}}$. 
Now, $(X_t)$ is a pure jump process, hence by applying It\^o formula, reading for any $g\in C^2$ and $(S_t)$ a semimartingale, $dg(S_t)= g'(S_{t-}) dS_t^c+ \Delta
g(S_t)$, where $(S^c_t)$ denotes the continuous component of $(S_t)$,  we get
\[d(X^2_{t} +\frac{c}{p/2} t) ^{p/2}=
 c(X^2_{t-} +\frac{c}{p/2} t) ^{p/2-1}dt + (X^2_{t} +\frac{c}{p/2} t) ^{p/2}
-(X^2_{t-} +\frac{c}{p/2} t) ^{p/2}.
\]
Writing the semimartingale decomposition for the process $(X_t^p)$, 
we remark that the hypothesis of the lemma implies that 
\[\BbE(dX_t^p|\cF_{t-})=
\Gamma f^p (\xi_{t-})  dt\leq -c X^{p-2}_{t-} dt \ \textrm{on the event } \{\tau_A\geq t\}.\]
Applying conditional Minkowski inequality and the supermartingale hypothesis, we get, on the set $\{\tau_A\geq t\}$,
\begin{eqnarray*}
\BbE((X^2_{t} +\frac{c}{p/2} t) ^{p/2}|\cF_{t-})
&\leq& [\BbE((X^p_{t} |\cF_{t-})^{2/p} +\frac{c}{p/2} t]^{p/2}\\
&=& [(X^p_{t-}+\BbE(dX^p_t |\cF_{t-})^{2/p} +\frac{c}{p/2} t]^{p/2}\\
\cut{&\leq& [X^p_{t-}-cX_{t-}^{p-2} dt)^{2/p} +\frac{c}{p/2} t]^{p/2}\\}
&\leq& ((X^2_{t-}(1-\frac{c}{ X_{t-}^2}dt)^{2/p}+\frac{c}{p/2} t ) ^{p/2}\\
\cut{&\leq& ((X^2_{t-}(1-\frac{c}{p/2}\frac{1}{ X_{t-}^2}dt)+\frac{c}{p/2} t ) ^{p/2}\\}
&\leq& (X^2_{t-}-\frac{c}{p/2}dt+\frac{c}{p/2} t ) ^{p/2}.
\end{eqnarray*}
Hence, on the event $\{\tau_A\geq t\}$, we have the estimate
\[\BbE(dZ_t |\cF_{t-})\leq c(X_{t-}^2 +\frac{c}{p/2} t)^{p/2-1}dt +
(X_{t-}^2 +\frac{c}{p/2} t -\frac{c}{p/2} dt)^{p/2}-
(X_{t-}^2 +\frac{c}{p/2} t )^{p/2}. \]
Simple expansion of the remaining differential forms (containing now only $\cF_{t-}$-measurable random variables)
yields $\BbE(dZ_t |\cF_{t-})\leq 0$.
\end{proof}

\begin{coro}
\label{cor:pgetwo}
Let $f\in\Dom_+(\Gamma)$   tending to infinity, $p\geq 2$,  and $a>0$.
Use the abbreviation
 $A:=\sublev{f}{a}$. 
 Denote $X_t=f(\xi_t)$ and assume further
that $f^p\in\Dom_+(\Gamma)$. 
If there exists $c>0$ such that
\[\Gamma f^p(x)\leq -cf^{p-2}(x), \forall x\not\in A, \]
then there exists $c'>0$ such that
\[ \BbE_x(\tau_A^q)\leq c' f(x)^{2q} \ \textrm{for all $q\leq p/2$ and all $x\in\BbX$.}\]
\end{coro}
\begin{proof}
Without loss of generality, we can assume that $x\in A^c$ since otherwise
the corollary holds trivially.
Denoting by $Y_t=X_{t\wedge\tau_A}^2 +\frac{c}{p/2} t\wedge\tau_A$, we observe that $Z_t=Y_t^{p/2}$ is a non-negative supermartingale by virtue of the lemma \ref{lem:pgetwo}. Since the function $\BbR_+\ni w
\mapsto w^{2q/p}\in \BbR_+$ is increasing and concave for $q\leq p/2$, it follows that $Y_t^q$ is also a supermartingale. Hence, 
\[\frac{c}{p/2} \BbE_x[(t\wedge\tau_A)^q]
\leq \BbE_x(Y_t^q)\leq \BbE_x (Y_0^q)= f(x)^{2q}.\]
We conclude by the dominated convergence theorem on identifying $c'=\frac{p}{2c}$.
\end{proof}

\begin{prop}
\label{pro:pletwo}
Let $f\in\Dom_+(\Gamma)$   tending to infinity, $0<p\leq 2$,  and $a>0$.
Use the abbreviation
$A:=\sublev{f}{a}$. Denote $X_t=f(\xi_t)$ and assume further
that $f^p\in\Dom_+(\Gamma)$. 
If there exists $c>0$ such that
\[\Gamma f^p(x)\leq -cf^{p-2}(x), \forall x\not\in A, \]
then the process, defined by 
$Z_t=X^p_{\tau_A\wedge t} +\frac{c}{q} (\tau_A\wedge t) ^{q}$,
 satisfies 
 \[\BbE_x(Z_t)\leq c'' f(x)^p, \ \textrm{for all}\ q\in]0, p/2].\]
\end{prop}
\begin{proof}
Since $d(X_t^p +\frac{c}{q} t^q)=dX_t^p+c t^{q-1}dt$, we have
\[\BbE(dZ_t|\cF_{t-})\leq \BbE(dZ_t|\cF_{t-})\id_{\{\tau_A\geq t\}}\leq
c dt \id_{\{\tau_A\geq t\}} (-X^{p-2}_{t-}+t^{q-1}).\]
Now, $\frac{q}{p}\leq \frac{1}{2}\leq\frac{1-q}{2-p}$. Choosing $v\in ] 
\frac{q}{p},\frac{1-q}{2-p}[$, we write
\begin{eqnarray*}
%\label{eq-diff-ineq}
\BbE(dZ_t|\cF_{t-})&\leq &c
dt \id_{\{\tau_A\geq t\}}(-X_{t-}^{p-2}+t^{q-1})\id_{\{X_{t-}\in]a, t^v]\}}\\
&&+ cdt \id_{\{\tau_A\geq t\}}(-X_{t-}^{p-2}+t^{q-1})\id_{\{X_{t-}\in]t^v,+\infty]\}}.
\end{eqnarray*}
For $X_{t-}\leq t^v$, the first term of the right hand side of the previous inequality is non-positive; as for the second, the condition $X_{t-}>t^v$ implies
that $-X_{t-}^{p-2}+t^{q-1}\leq t^{q-1}$. Combining the latter with Markov inequality guarantees 
that, on the set $\{X_{t-}>t^v\}$, we have
$\BbE_x(dZ_t)\leq c dt \id_{\{\tau_A\geq t\}} \frac{f(x)^p}{t^{vp}}t^{q-1}.$
Integrating this differential inequality  yields
$\BbE_x(Z_t)\leq
 c f^p(x)\int_{a^{1/v}}^{\infty} \frac{dt}{t^{vp+1-q}}$; the condition $v>q/p$
insures the finiteness of the last integral proving thus the 
lemma with $c''=c\int_{a^{1/v}}^{\infty} \frac{dt}{t^{vp+1-q}}$.
\end{proof}

\begin{coro}
\label{cor:pletwo}
Under the same conditions as in proposition \ref{pro:pletwo},
there exists $c'''>0$ such that 
$\BbE_x (\tau_A^q) \leq c''' f(x)^p, \forall q\in]0, p/2].$
\end{coro}
\begin{proof}
Since $X_t$ is non-negative, $\frac{q}{c}Z_t\geq (t\wedge\tau_A)^q$.  
By the previous proposition, 
$\BbE_x[(t\wedge\tau_A)^q]\leq \frac{q}{c}\BbE_x(Z_t)\leq
c'' \frac{q}{c}  f(x)^p$. We conclude by the dominated convergence theorem on identifying $c'''= \frac{c''q}{c}$.
\end{proof}

\begin{rema}{}
All the propositions, lemmata, and corollaries shown so far allow to prove statement 1 of the theorem
\ref{thm:moments}.
The subsequent propositions are needed for statement 2 of this theorem.
Notice also that the following proposition \ref{pro:gener-foster} is very important and tricky. It provides us with a generalisation of the theorem 3.1 of Lamperti \cite{Lam2} and serves twice in this paper: one first time to establish conditions for some moments of passage time to be infinite
(statement 2 of theorem \ref{thm:moments}) and once more in a very different context, namely for finding conditions for the chain not to implode (statement 2 of theorem  \ref{thm:implo}).
\end{rema}
\begin{prop}
\label{pro:gener-foster}
Let  $(\Omega, \cG,(\cG_t)_t, \BbP)$ be a filtered probability space and 
$(Y_t)$ be a $(\cG_t)$-adapted process taking values in an unbounded subset of $\BbR_+$.
Let $a>0$ and $T_a=\inf\{t\geq 0: Y_t\leq a\}$.
Suppose that there exist constants $c_1>0$ and $c_2>0$ such that
\begin{enumerate}
\item
$\BbE(dY_t|\cG_{t-})\geq - c_1 dt$ on the event $\{T_a>t\}$, and 
\item there exists $r>1$ such that
$\BbE(dY^r_t|\cG_{t-})\leq  c_2 Y_{t-}^{r-1} dt$ on the event $\{T_a>t\}$.
\end{enumerate}
Then, for all $\alpha\in]0,1[$, there exist $\epsilon>0$ and $\delta>0$ such that 
\[\forall t>0: \BbP(T_a>t+\epsilon Y_{t\wedge T_a}| \cG_t)\geq 1-\alpha, \ \textrm{on the event}\ 
\{T_a>t; Y_t>a(1+\delta)\}.\]
\end{prop}

\begin{rema}{}
The meaning of the previous proposition \ref{pro:gener-foster} is the following. If the process $(Y_t)$ has average increments bounded from below by a constant $-c_1$, it is intuitively appealing to suppose that the average time
of reaching $0$ is of the same order of magnitude as $Y_0$. However this intuition proves false if the increments are wildly unbounded since then $0$ can be reached in one step. The condition 2, by imposing some control on $r$-moments of the increments with $r>1$,  prevents this from occurring. 
It is in fact established that if $T_a>t$, the remaining time $T_a-t$ to reach $A_a:=[0,a]$ exceeds $\epsilon Y_t$ with probability $1-\alpha$; more precisely, for every $\alpha$ we can chose $\epsilon$ such that $\BbP(T_a-t>\epsilon Y_t|\cG_t)\geq 1-\alpha$.
\end{rema}

  \vskip3mm
\noindent\textit{Proof of  proposition \ref{pro:gener-foster}.}
Let $\sigma=(T_a-t)\id_{\{T_a\geq a\}}$; then for all $s>0$ we have
$\{\sigma<s\}=\{T_a\geq t\}\cap\{T_a<t+s\}\in \cG_{t+s-}$.
To prove the proposition, it is enough to establish 
$\BbP(\sigma >\epsilon Y_t|\cG_t) \geq 1-\alpha$ on the set 
$\{\sigma>0; Y_t>a(1+\delta)\}$.
On this latter set: $\BbP(\sigma>\epsilon Y_t|\cG_t)= 
\BbP(Y_{t+(\epsilon Y_t)\wedge \sigma}>a|\cG_t)$, because once 
the process $Y_{t+(\epsilon Y_t) \wedge \sigma}$ enters in $A_a$, it remains there for ever, due to the stopping by $\sigma$.
 On defining
$U:=Y_{(\epsilon Y_t)\wedge\sigma+t}$ one has 
\begin{eqnarray*}
\BbE(U|\cG_t)&=&\BbE(U\id_{\{U\leq a\}}|\cG_t)+\BbE(U\id_{U> a\}}|\cG_t)\\&\leq& a+(\BbE(U^r|\cG))^{1/r}(\BbP(U>a|\cG_t))^{1-1/r};
\end{eqnarray*}
 therefore
\[\BbP(U>a|\cG_t))\geq
\left[\frac{(\BbE(U|\cG_t)-a)_+}{(\BbE(U^r|\cG))^{1/r}}\right]^{r/(r-1)}.\]
To minorise the numerator, we observe that
\[\BbE(U|\cG_t)=\BbE(Y_{t+\epsilon Y_t} -Y_t|\cG_t) +Y_t=
\int_t^{t+\epsilon Y_t} \BbE(dY_s|\cG_t) +Y_t\geq -c_1 \epsilon Y_t +Y_t.\]
To majorise the denominator $\BbE(U^r|\cG)=
\BbE(Y^r_{t+(\epsilon Y_t)\wedge \sigma}|\cG_t)$, we must be able to majorise 
$\BbE(Y^r_{t+s\wedge\sigma}|\cG_t)$ for arbitrary $s>0$. 
Let $t>0$ be arbitrary and $S$ be a $\cG_t$-optional random variable, $S>0$.
For $c_3=c_2/r$ and 
any $s\in]0,S]$, define  
\[F_S(s)=\BbE[(Y_{t+s\wedge\sigma} + c_3S-c_3s\wedge\sigma)^r|\cG_t].\] 
We shall show that $F_S(s)\leq F_S(s-)$ for all $s\in]0,S]$.  
It is enough to show this inequality on $\{\sigma>s\}$ since otherwise
$F_S(s)=F_S(s-)$ and there is nothing to prove. 
To show that $F_S$ is decreasing in $]0,S]$, it is enough to show 
that $\BbE(d\Xi_s|\cG_{t+s-})\leq 0$ for all $s\in]0,S]$, where
$\Xi_s=(Y_{t+s\wedge\sigma} + c_3S-c_3s\wedge\sigma)^r$.
Now, on $\{\sigma>s\}$, by use of It\^o's formula, we get
\[d\Xi_s= -rc_3(Y_{t+s-}+c_3 S-c_3 s)^{r-1} ds
+(Y_{t+s} + c_3S-c_3s)^r
-(Y_{t+s-} + c_3S-c_3s)^r.\]
Moreover, using  Minkowski inequality, we get
\[\BbE[(Y_{t+s} + c_3S-c_3s)^r|\cG_{t+s-}]\leq
\left[\BbE(Y_{t+s}^r|\cG_{t+s-})^{1/r}
+c_3S-c_3s\right]^r,\]
and by use of the hypothesis
\[\BbE(Y^r_t|\cG_{t-})\leq Y^r_{t-} +c_2 Y^{r-1}_{t-}  \id_{\{\rho>t\}} dt=
 Y^r_{t-}\left(1+ \frac{c_2}{Y_{t-}}\id_{\{T_a>t\}}dt\right)\leq 
 \left(Y_{t-}+c_3\id_{\{T_a>t\}}dt\right)^r. \]
Therefore, $\BbE(d\Xi_s|\cG_{t+s-})\leq 0$ for all $s\in]0,S]$.
Subsequently, for all $S>0$; 
\[Y_{t+S}^r=F_S(S)\leq \lim_{s\rightarrow 0+} F_S(s)=(Y_t+c_3 S)^r.\] 
Since $S$ is an arbitrary $\cG_t$-optional random variable, on choosing $S=\epsilon Y_t$, we get finally
\[(\BbE(Y^r_{t+\epsilon Y_t}|\cG_t))^{1/r}\leq  Y_t+c_3\epsilon Y_t.\]
Substituting yields, that   for any $\alpha\in]0,1[$,  parameters 
$\epsilon>0$ and $\delta>0$ can be chosen so that the following inequality holds: 
\[\BbP(U>a|\cG_t)\geq \left[\frac{(1-c_1\epsilon) -\frac{a}{Y_t})_+}{1-c_3\epsilon}\right]^{\frac{r}{r-1}}\geq (1-c_1\epsilon -\frac{1}{1+\delta})_+^{\frac{r}{r-1}} (1+c_3\epsilon)^{-\frac{r}{r-1}} \geq 1-\alpha.\]
\eproof

\begin{lemm}
\label{lem:y-ltbz}
Let $(\Omega, \cG, (\cG_t), \BbP)$ be a filtered probability space,
$(Y_t)$ and $(Z_t)$ two $\BbR_+$-valued, $(\cG_t)$-adapted processes on it.
For $a>0$, we denote $S_a=\inf\{t\geq 0: Y_t\leq a\}$
and $T_a=\inf\{t\geq 0: Z_t\leq a\}$. Suppose that there exist positive constants
$a,b,p,K_1$ such that
\begin{enumerate}
\item $Y_t\leq b Z_t$ almost surely, for all $t$, and
\item
$\BbE(Z_{t\wedge T_a}^p)\leq K_1$.
\end{enumerate}
Then, there exists $K_2>0$ such that
$\BbE(Y_{t\wedge S_{ab}}^p)\leq K_2$.
\end{lemm}
\begin{proof}
For arbitrary $s>0$, the condition $Z_s<a$ implies $X_s\leq bZ_s<ab$ almost surely. Hence, $\{S_{ab}\geq t\} \subseteq \{T_a\geq t\}$. Then,
\begin{eqnarray*}
\BbE(Y^p_{t\wedge S_{ab}})&\leq & \BbE\left[Y^p_{t\wedge S_{ab}}(\id_{\{S_{ab}\geq t\}}+\id_{\{S_{ab}> t\}})\right]\\
&\leq& \BbE(Y_t^p \id_{\{S_{ab}\geq t\}}) +(ab)^p\leq b^p K_1+(ab)^p:=K_2.
\end{eqnarray*}
\end{proof}

\begin{prop}
\label{pro:non-existence-moments}
Let $(\Omega, \cG, (\cG_t), \BbP)$ be a filtered probability space,
$(Y_t)$ and $(Z_t)$ two $\BbR_+$-valued, $(\cG_t)$-adapted processes on it.
For $a>0$, we denote $S_a=\inf\{t\geq 0: Y_t\leq a\}$
and $T_a=\inf\{t\geq 0: Z_t\leq a\}$. 
Suppose that 
\begin{enumerate}
\item there exist positive constants
$a,c_1, c_2, r$ such that
\begin{itemize}
\item $\BbE(dZ_t^2|\cG_{t-})\geq -c_1dt$ on the event $\{T_a\geq t\}$,
\item
$\BbE(dZ_t^{r}|\cG_{t-})\leq c_2 Z_{t-}^{r-1}dt$ on the event $\{T_a\geq t\}$,
\item $Z_0=z>a$.
\end{itemize}
\item $Y_0=y$ and there exists a constant $b>0$ such that
\begin{itemize}
\item $ab< y <bz$ and
\item $Y_t\leq b Z_t$ almost surely for all $t$.
\end{itemize}
\end{enumerate}
If for some $p>0$, the process $(Y^{p}_{t\wedge S_{ab}})$ is a submartingale, then $\BbE(T_a^q)=\infty$ for all $q>p$.
\end{prop}
\begin{proof}
We can without loss of generality examine solely the case 
$\BbP(S_{ab}<\infty)=1$, since otherwise $\BbE(T_a^q)=\infty$ holds trivially for all $q>0$.
Assume further that for some $q>p$, it happens $\BbE(T_a^q)<\infty$.
Hypothesis 1 allows applying proposition \ref{pro:gener-foster};  for $\alpha=1/2$, we can thus determine positive constants $\epsilon$ and $\delta$ such that 
$\BbP(T_a>t+\epsilon Z_{t\wedge T_a}|\cG_t) \geq 1/2$ on the event 
$\{Z_{t\wedge T_a}>a(1+\delta)\}$. Hence
\begin{eqnarray*}
\BbE(T_a^q) &\geq & \BbE(T_a^q \id_{\{Z_{t\wedge T_a} >a (1+\delta)\}})\\
&\geq & \frac{1}{2} \BbE\left[ (t+Z_{t\wedge T_a})^q  \id_{\{Z_{t\wedge T_a} >a (1+\delta)\}}\right]\\
&\geq & \frac{\epsilon^q}{2} \BbE(Z^{q}_{t\wedge T_a})- 
 \frac{\epsilon^q}{2} a^{q} (1+\delta)^{q}.
\end{eqnarray*}
Now, finiteness of the $q$ moment of $T_a$ implies the existence
of a constant $K_1>0$ such that $\BbE(Z^{q}_{t\wedge T_a})\leq K_1$.
From the previous lemma \ref{lem:y-ltbz} follows that there exists some $K_2>0$ such that $\BbE(Y^{q}_{t\wedge S_{ab}})\leq K_2$. 
Since the time $S_{ab}$ is assumed almost surely finite,
$\lim_{t\rightarrow\infty}  \BbE(Y^{q}_{t\wedge S_{ab}})=
\BbE(Y^{q}_{S_{ab}})\leq (ab)^{q}$. On the other hand, 
$(Y^{p}_{t\wedge S_{ab}})$ is a submartingale, so is thus a fortiori
$(Y^{q}_{t\wedge S_{ab}})$. Hence, $\BbE(Y^{q}_{t\wedge S_{ab}})\geq
\BbE(Y^{q}_0)=y^{q}$, leading to a contradiction if we chose $y>ab$.
\end{proof}

\begin{coro}
\label{cor:non-existence}
Let $(\Omega, \cG, (\cG_t), \BbP)$ be a filtered probability space,
$(X_t)$ a $\BbR_+$-valued, $(\cG_t)$-adapted processes on it.
For $a>0$, we denote  $T_a=\inf\{t\geq 0: X_t\leq a\}$. 
Suppose that there exist positive constants
$a,c_1, c_2, p, r$ such that
\begin{itemize}
\item $X_0=x>a$,
\item $\BbE(dX_t|\cG_{t-})\geq -c_1dt$ on the event $\{T_a\geq t\}$,
\item
$\BbE(dX_t^{r}|\cG_{t-})\leq c_2 X_{t-}^{r-1}dt$ on the event $\{T_a\geq t\}$,
\item
$(X^p_{t\wedge T_a})$ is a submartingale.
\end{itemize}
Then $\BbE(T_a^q)=\infty$ for all $q>p$.
\end{coro}

After all this preparatory work, the proof of the theorem  \ref{thm:moments}
is now immediate. 

  \vskip3mm
\noindent\textit{Proof of the theorem \ref{thm:moments}.} Write $X_t=f(\xi_t)$ and use the abbreviation $A:=\sublev{f}{a}$.
Notice  moreover that $\tau_{A}=T_a$. 
\begin{enumerate}
\item
Since $f\rightarrow \infty$ the set $A$ is finite.
The integrability of the passage time follows from corollaries
\ref{cor:pgetwo} and \ref{cor:pletwo}.
\item On identifying  $Z_t$ and $Y_t$ in proposition \ref{pro:non-existence-moments} with  $g(\xi_t)$ and $f(\xi_t)$ respectively, we see that the conditions of the theorem imply the hypotheses of the proposition. The non-existence of moments immediately follows.
\end{enumerate}
\eproof

%%%%%%%%%%%%%%%%

\subsection{Proof of  theorems on explosion and implosion}
\label{ssec:explo}
As stated in the introduction, the result concerning integrability of explosion time is reminiscent to Foster's criterion for positive recurrence! The reason lies in the  lemma \ref{lem:expec-time}  and
the fact that theorem \ref{thm:explo}, establishing explosion, is equivalent to the proposition \ref{pro:foster} below.

Since we wish to treat explosion, we work with a Markov chain evolving on the augmented state space
$\hat{\BbX}=\BbX\cup\{\partial\}$ and with augmented functions $\hat{f}$ as explained in the introduction. Note that if $f>0$ then $\hat{f}\rt{\BbX}=f>0$ while $\hat{f}(\partial)=0$. We extend also naturally $X_t=f(\xi_t)$ (and \textit{stricto sensu} also the
generator $\Gamma$). We use the ``hat'' notation  $\hat{\cdot}$ to denote quantities referring to the extended chain. 

\begin{prop}
\label{pro:foster}
Let $\hat{f}:\hat{\BbX}\rightarrow [0,\infty[$ belong to $\Dom_+(\hat{\Gamma})$
and verify $\hat{f}(x)>0$ for all $x\in\BbX$ (it vanishes only at $\partial$).
We denote $f$ the restriction of $\hat{f}$ on $\BbX$.
The following are equivalent:
\begin{enumerate}
\item There exists $\epsilon >0$ such that 
$\Gamma f(x) \leq -\epsilon$ for all $x\in\BbX$.
\item $\hat{\BbE}_x(\hat{\tau}_\partial)<\infty$ for all $x\in\BbX$.
\end{enumerate}
\end{prop}

\begin{proof}
Let $Y_t=\hat{f}(\hat{\xi}_t)$, for $t\geq 0$. Then $(Y_t)$ is a 
$[0,\infty[$-valued, $(\hat{\cF}_t)$-adapted process. Note also
that $\hat{\tau}_\partial=T_0:=\inf\{t\geq 0: Y_t=0\}$
\begin{description}
\item{[$1\Rightarrow 2$:]} 
The condition $\Gamma f(x) \leq -\epsilon$ for all $x\in\BbX$
is equivalent to the condition $\BbE(dY_{t\wedge T_0}|\cG_{t-})
\leq -\epsilon \id _{T_0\geq t} dt$. The conclusion follows from direct
application of the lemma \ref{lem:expec-time}.
\item{[$2\Rightarrow 1$:]} 
Define $\hat{f}(x)=\epsilon\BbE_x (\hat{\tau}_\partial)$. Obviously
$\hat{f}(\partial)=0$ while $0<\hat{f}(x)<\infty$ for all $x\in\BbX$.
Conditioning on the first move of the embedded Markov chain $\tilde{\xi}$ we get:
\begin{eqnarray*}
m_f(x)&=&\BbE(f(\tilde{\xi}_{1})-f(x)|\tilde{\xi}_0=x)=\BbE(f(\xi_{J_1})-f(\xi_0)|\xi_0=x)\\
&=&\epsilon \sum_{y\in\BbX\setminus\{x\}} P_{xy} \BbE_y(\hat{\tau}_\partial)-\epsilon\BbE_x(\hat{\tau}_\partial)
=\epsilon  (\BbE_x(\hat{\tau}_\partial)-\BbE_x(\sigma_1))-\epsilon\BbE_x(\hat{\tau}_\partial)\\
&=&-\epsilon \BbE_x(\sigma_1)
=-\frac{\epsilon}{\gamma_x}.
\end{eqnarray*}
Hence $\Gamma f(x) \leq -\epsilon$ for all $x\in\BbX$. \eproof
\end{description}

\vskip3mm
\noindent\textit{Proof of the theorem \ref{thm:explo}.}
Evident by proposition \ref{pro:foster}, since $\zeta=\hat{\tau}_\partial$.
\end{proof}

\vskip3mm
\noindent\textit{Proof of the proposition \ref{pro:explo}.}
Let  $G(z)=\int_0^z \frac{dy}{g(y)}$. Then $G$ is differentiable, with $G'(z)=\frac{1}{g(z)}>0$ hence an increasing function of $z\in[0,b]$. Since $g$ is increasing, $G'$ is decreasing and hence $G$ is concave satisfying $\lim_{z\rightarrow 0}G(z)=0$ and  $\lim_{z\rightarrow \infty} G(z)<\infty$. Additionally, boundedness of $G$ imply that $G\circ f\in \ell^1_+(\Gamma)$.
Due to differentiability and concavity of $G$, we have:
\begin{eqnarray*}
\Gamma G\circ f(x)&= &\gamma_x\BbE\left[G(f(\tilde{\xi}_n)+\Delta_{n+1}^f)-G(f(\tilde{\xi}_n))\vert\tilde{\xi}_n=x\right]\\
&\leq& 
\cut{
\gamma_x G'(f(x))\BbE\left[\Delta_{n+1}^f\vert\tilde{\xi}_n=x\right]\\
&=&}
 \gamma_x \frac{1}{g(f(x))} m_f(x)=\frac{\Gamma f(x)}{g(f(x))}  \leq -c;
\end{eqnarray*}
we conclude by theorem \ref{thm:explo} because $G\circ f$ is strictly positive and bounded.
\eproof

\vskip3mm
\noindent\textit{Proof of the theorem \ref{thm:cond-explo}.}
The conditions of the theorem imply transience of the chain. In fact, let $Y_t=f(\xi_{t\wedge\tau_A})$. The condition $\Gamma f(x)\leq -\epsilon$ on $A^c$ implies that $(Y_t)$ is  a non-negative supermartingale, converging almost surely towards $Y_\infty$. For every $x\in A^c$, Fatou's lemma implies
$\BbE_x(Y_\infty)\leq \lim_{t\rightarrow \infty} \BbE_x(Y_t)\leq f(x)$. Suppose now that the chain is not transient, hence $\BbP_x(\tau_A<\infty)=1$. In that case
$\lim_{t\rightarrow \infty}Y_t =Y_{\tau_A}$ almost surely. We get then by monotone convergence theorem $\inf_{y\in A} f(y) \leq \BbE_x(f(\xi_{\tau_A}))\leq f(x)$, in contradiction with the hypothesis that there exists $x\in A^c$ such that $f(x)<\inf_{y\in A} f(y)$. Hence, $\BbP_x(\tau_A=\infty)>0$.

Let now $T:=\tau_A\wedge \zeta$;  on defining $Z_t=\hat{f}(\xi_{t\wedge T})$, we obtain  then that $\BbE(dZ_t|\cF_{t-}) \leq -\epsilon \id_{\{T>t\}} dt$. We conclude by proposition \ref{pro:foster} and the fact that the event $\{\tau_A=\infty\}$ is not negligible. 
\eproof

\vskip3mm
\noindent\textit{Proof of the theorem \ref{thm:non-explo}.}
Let $G(z)=\int_0^z \frac{1}{g(y)} dy$; this function is differentiable with $G'(z)=\frac{1}{g(z)}>0$, hence increasing. Since $g$ is increasing, $G'$ is decreasing, hence $G$ is concave. Non integrability of the infinite tail means that  $G$  is unboundedly increasing towards $\infty$ as $z\rightarrow \infty$, 
hence $G\circ f \rightarrow \infty$. 
Concavity and differentiability of $G$ imply that $G(f(x)+\Delta)-G(f(x))\leq \Delta G'(f(x))$;   integrability of $f$ guarantees then $0\leq \BbE(G(f(\tilde{\xi}_n)+\Delta_{n+1}^f)|\tilde{\xi}_n=x)\leq G(f(x))+ \frac{m_f(x)}{g(f(x))}<\infty$ so that 
 $F:=G\circ f\in\ell^1_+(\Gamma)$ as well. Now
\begin{eqnarray*}
\Gamma G\circ f(x)&= &\gamma_x\BbE\left[G(f(\tilde{\xi}_n)+\Delta_{n+1}^f)-G(f(\tilde{\xi}_n))\vert\tilde{\xi}_n=x\right]\\
&=&\gamma_x\BbE\left[\int_{f(x)}^{f(x)+\Delta_{n+1}^f} \frac{1}{g(y)} dy|\tilde{\xi}_n=x\right]\\
&\leq& 
\cut{
\gamma_x\BbE\left[\frac{1}{g(f(x))} \Delta_{n+1}^f \id_{ \Delta_{n+1}^f \geq 0}-
\frac{1}{g(f(x))} |\Delta_{n+1}^f| \id_{ \Delta_{n+1}^f < 0}|\tilde{\xi}_n=x\right]\\
&=&
}
\gamma_x\frac{1}{g(f(x))} m_f(x) \leq c.
\end{eqnarray*}
Let $X_t=F(\xi_t)$ be the process obtained form the Markov chain after transformation by $F$. Using the semimartingale decomposition of $X_t=X_0+M_t+\int_{]0,t]} \Gamma F(\xi_{s-}) ds$, it becomes obvious then that $\BbE_x X_t\leq 
F(x) +ct$, showing that for every finite $t$, $\BbP_x(X_t=\infty)=0$. But since $F\rightarrow \infty$ the process itself $\xi_t$ cannot explode.

\eproof

%\subsection{Implosion}
%\label{ssec:implo}

\begin{prop}
\label{pro:implosion-irreducibility}
Suppose the chain is recurrent and there exists a finite set $A\in\cX$ and a constant $C>0$ such that, for all $x\in\BbX$, the uniform bound  $\BbE_x\tau_A\leq C$ holds. Then the chain implodes towards any state $z\in\BbX$.
\end{prop}
\begin{proof} 
First remark that obviously $\sigma_0\leq\tau_A$, where $\sigma_0$ is the holding time at the initial state, i.e.\
 $\frac{1}{\gamma_x}= \BbE_x(\sigma_0)\leq \BbE_x\tau_A\leq C$. Hence $\underline{\gamma}:=\inf_{x\in\BbX}{\gamma_x} >0$.
Let $a$ be an arbitrary point in $A$ and $z$ an arbitrary state in $\BbX$. Irreducibility means that there exists 
a path of finite length, say $k$, satisfying $a\equiv x_0, \ldots, x_k\equiv z$ and $P_{x_0,x_1} \cdots P_{x_{k-1}, x_k}=\delta_a >0$.  Then, for $K'> \frac{1}{2\underline{\gamma}}\sup_{a\in A} k_a$, we estimate 
\[\BbP_a(\tau_z\leq K') \geq  \BbP_a 
(\sum_{i=0}^{k_a-1} \sigma_i \leq K' |
\tilde{\xi}_1=x_1,\ldots, \tilde{\xi}_k= x_k=z) 
\BbP_a   (\tilde{\xi}_1=x_1,\ldots, \tilde{\xi}_k= x_k=z).\]
Now, Markov's inequality yields
$\BbP_a(\tau_z\leq K') \geq (1-\frac{k_a}{K'\underline{\gamma}})\delta_a\leq \frac{1}{2} \inf_a\delta_a:=\delta'>0$.
Using exactly the same arguments, we show that $\BbP_a(\tau_{A^c} \leq K'')\geq \delta''>0$ for some parameters $K''<\infty$ and $\delta''>0$. 

Denote $K=\max(K', K'')<\infty$,  $\delta=\min(\delta', \delta'')>0$, and 
$r=\sup_{a\in A} \BbE_a(\tau_z)$.
The hypothesis of the proposition  implies
\begin{eqnarray*}
\BbE_x\tau_z &=& \BbE_x(\tau_z|\tau_z\leq \tau_A)\BbP_a(\tau_z\leq \tau_A)+\BbE_x(\tau_z|\tau_z> \tau_A)\BbP_a(\tau_z> \tau_A)\\
&\leq & C +\BbE_x(\tau_z|\tau_z> \tau_A)\BbP_a(\tau_z> \tau_A).
\end{eqnarray*}
Now $\BbE_x(\tau_z|\tau_z>\tau_A) \leq \frac{\BbE_x \tau_A}{\BbP_x(\tau_z>\tau_z)} +\BbE_x (\BbE(\tau_z-\tau_A|\cF_{\tau_A})| \tau_z<\tau_A)$.
We get finally $\BbE_x \tau_z \leq 2C + r$.

On the other hand, 
\begin{eqnarray*}
\BbE_a(\tau_z) &=& \BbE_x(\tau_z|\tau_z\leq K)\BbP_a(\tau_z\leq K)+\BbE_x(\tau_z|\tau_z> K)\BbP_a(\tau_z> K)\\
&=& K + \BbE_a (\BbE_{X_K} (\tau_z)).
\end{eqnarray*}
Now, on the set $\{\tau_z>K\}$,  $X_K$ can be any $a'\in A$ or any $x'\in A^c\setminus \{z\}$. Hence, $\BbE_a(\tau_z) \leq K + \delta \max\left(\sum_{x'\in A^c\setminus\{z\}}\BbE_{x'}\tau_z \BbP_a(X_K=x'), 
 \sum_{a'\in A}\BbE_{a'}\tau_z \BbP_a(X_K=a')\right)$.
 Using the previously obtained majorisations, we get
 $\BbE_a \tau_z \leq K + \delta (2C + r)$ and taking the $\max_a$ of the l.h.s.\ we obtain finally $r(1-\delta) \leq K +2C \delta$ proving thus that
 $\sup_x \BbE_x\tau_z \leq 2C + \frac {K+2C}{1-\delta}$ that guarantees implosion towards any $z\in\BbX$.
 \end{proof}

\vskip3mm
\noindent\textit{Proof of the theorem \ref{thm:implo}.}
\begin{enumerate}
\item 
We first show implosion. Since $f$ is bounded, the condition $\Gamma f(x)=\gamma_x m_f(x) \leq -\epsilon$  guarantees that $\underline{\gamma}=\inf_x \gamma_x>0$.
\begin{description}
\item{[$\Rightarrow$:]}
The condition $\Gamma f(x)\leq -\epsilon$ for $x\not\in \sublev{f}{a}$ guarantees --- by virtue of the 
  proposition \ref{pro:foster} --- that $\BbE_x \tau_{\sublev{f}{a}}\leq \frac{f(x)}{\epsilon}\leq \sup_{z\not\in \sublev{f}{a}}\frac{ f(z)}{\epsilon}$.  Abbreviate $B:=\sublev{f}{a}$ and remark that $B$ is finite by hypothesis. Now, due to recurrence and irreducibility,  if there exists a constant $C$ such that $x\not\in B \Rightarrow \BbE_x\tau_B<C$, then, for every $z\in\BbX$, there exists a constant $C'$ such that $x\ne z\Rightarrow \BbE_x \tau_{\{z\}}<C'$, by virtue of the previous proposition \ref{pro:implosion-irreducibility}.
  \item{[$ \Leftarrow$:]} Suppose now that for a finite $A\in\cX$, there exists a constant $C$ such that $\BbE_x \tau_A\leq C$. Consequenlty $\frac{1}{\gamma_x} \leq \BbE_x \tau_A \leq C$, leading necessarily to the lower bound $\underline{\gamma}>0$. Define 
  \[f(x)=\left\{\begin{array}{ll} 0 & \textrm{if }\ x\in A\\
  \BbE_x \tau_A & \textrm{if }\ x\not\in A.
  \end{array} \right.\]
  Then it is immediate to show that $A=\sublev{f}{1}$ and that 
  for $x\not \in A$, we have $\Gamma f(x) \leq -1$. 
  \end{description}
 \item
 To show non-implosion, use 
 lemma \ref{lem:expec-time} to guarantee that if at some time $t$ the process $\xi_t$ is at some point  $x_0$ sufficiently large, then the time needed for the  process $X_t=f(\xi_t)$ to reach $\sublev{f}{a}$ exceeds $\epsilon f(x_0)$ with some substantially large probability. More precisely,  there exists $\alpha \in ]0,1[$ such that $\BbP_{x_0}(\tau_{\sublev{f}{a}}-t >\epsilon f(x_0))\geq 1-\alpha$. Therefore
$\BbE_{x_0} (\tau_{\sublev{f}{a}})\geq (1-\alpha)\epsilon f(x_0)$ and since $f\rightarrow \infty$ then this expectation cannot be bounded uniformly in $x_0$.
\end{enumerate}
 \eproof

\vskip3mm
\noindent\textit{Proof of the proposition \ref{pro:implosion}.}
Let $G(z)=\int_0^z \frac{dy}{g(y)}$. Since $G'(z)=\frac{1}{g(z)}>0$ the function $G$ is increasing with $G(0)=0$ and $G(b)=B$. Since $g$ is increasing $G'=\frac{1}{g}$ is decreasing, hence the function $G$ is concave. Then concavity leads to the majorisation 
\[m_{G\circ f} (x) \leq G'(f(x)) m_f(x)= \frac{m_f(x)}{g(f(x))}.\]
The condition imposed on the statement of the proposition implies that
$\Gamma G\circ f (x) \leq -1$.
We conclude from lemma \ref{lem:expec-time}.
\eproof

%%%%%%%%%%%%%%%%%%%%%%%%%%%%%%%

\section{Application to some critical models}
\label{sec:exam}

This section intends to treat some problems that even with without
the explosion phenomenon have  critical behaviour illustrating thus how our methods can be applied and showing their power.   

We need first some technical conditions.
Let $f\in\ell^{2+\delta_0}$ for some $\delta_0>0$. For every $g:\BbR_+\rightarrow \BbR_+$ a $C^3$ function, we get
\[m_{g\circ f}(x)=\BbE(g(f(\tilde{\xi}_{n+1}))-g(f(\tilde{\xi}_{n}))|\tilde{\xi}_{n}=x)=
g'(f(x))m_f(x)+\frac{1}{2}g''(f(x))v_f(x)+R_g(x),\]
where $R_g(x)$ is the conditional expectation of the remainder occurring in the
Taylor expansion\footnote{The most convenient form of the remainder is the Roche-Schlömlich one.}.
\begin{defi}{}
\label{def:R}
Let $F=(F_n)_{n\in\BbN}$ be an exhaustive nested sequence
of  sets $F_n\uparrow \BbX$, $f\in \ell_+^{2+\delta}(\Gamma)$ and $g\in C^3(\BbR_+;\BbR_+)$.  
We say that the chain (or its jumps) satisfies the \textit{remainder condition} (or shortly \textit{condition~R})
for $F,f$, and $g$, if
 \[\lim_{n\rightarrow\infty} \sup_{x\in F_n^c} R_g(x)/[g'(f(x))m_f(x)+\frac{1}{2}g''(f(x))v_f(x)]=0.\]
\end{defi}
The quantity $D_g(f,x):=g'(f(x))m_f(x)+\frac{1}{2}g''(f(x))v_f(x)$ in the expression above is the \textit{effective average drift}  at the scale defined by the function $g$. If the function $f$ is non-trivial, there exists a natural exhaustive nested sequence determined by the sub-level sets of $f$. When we omit to specify the nested sequence, we shall always consider the natural  one.

Introduce the notation $\ln_{(0)}s=s$ and recursively, for
all integers $k\geq 1$, $\ln_{(k)} s=\ln(\ln_{(k-1)} s)$ and denote by 
$L_k(s)=\prod_{i=0}^k \ln_{(i)} s$, for $k\geq 0$, and $L_k(s):=1$ for $k<0$.
Equivalently, we define $\exp_{(k)}$ as the inverse function of $\ln_{(k)}$.
In most occasions we shall use (Lyapunov) functions 
$g(s):=\ln_{(l)}^\eta s$
with some integer $l\geq 0$ and some real $\eta\ne0$, defined for $s$ sufficiently large, $s\geq s_0:= \exp_{(l)}(2)$ say. It is cumbersome but straightforward to show then that for 
$f\in\ell^{2+\delta_0}(\Gamma)$ with some $\delta_0>0$, the condition R is satisfied.
It will be shown in this section  that condition R and Lyapunov functions of the aforementioned form play a crucial role in the study of models where the effective average drift $D_g(f,x)$ tends to 0 in some controlled way with $n$ when $x\in F_n^c$, where $(F_n)$ is an exhaustive nested sequence;  models of this type lie in the deeply critical regime between recurrence and transience.

%%%%%%%% LAMPERTI %%%%%%%%%%%%%%%%
\subsection{Critical models on denumerably infinite and unbounded subsets of $\BbR_+$}
\label{ssec:lamperti}
Consider a discrete time irreducible Markov chain $(\tilde{\xi}_n)$  on 
a denumerably infinite and unbounded subset $\BbX$ of $\BbR_+$. Since now $\BbX$ inherits the composition properties stemming from the field structure of $\BbR$, we can define directly $m(x):=m_{\textrm{id}}(x)$ and 
$v(x):=v_{\textrm{id}}(x)$, where $\textrm{id}$ is the identity function on $\BbX$.
The model is termed \textit{critical} when the drift $m$ tends to 0, in some
precise manner, when $x\rightarrow\infty$. 

In \cite{Lam1} Markov chains on $\BbX=\{0,1,2,\ldots\}$ were considered; it has been established that
the chain is recurrent if $m(x)\leq v(x)/2x$ while is transient if
$m(x)> \theta v(x)/2x$ for some $\theta >1$. In particular, if $m(x)=\cO(\frac{1}{x})$ and $v(x)=\cO(1)$ the model is in the  critical regime.

The case with arbitrary degree of criticality
%%%%%%%%%%%%%%%%%%%% 
 \[m(x)=\sum_{i=0}^k \frac{\alpha_i} {L_i(x)}+o\left(\frac{1} {L_k(x)}\right)
 \
 \textrm{and}\ 
v(x)=\sum_{i=0}^k \frac{x\beta_i} {L_i(x)}+o\left(\frac{x} {L_k(x)}\right),\] with $\alpha_i, \beta_i$ constants, has been settled in \cite{MenAsyIas} by using  Lyapunov functions.
Under the additional conditions 
\[\limsup_{n\rightarrow\infty} \tilde{\xi}_n=\infty\ \textrm{and}\ \liminf_{x\in\BbX} v(x)>0\] and some technical moment conditions --- guaranteeing the condition R for this model --- that can straightforwardly be shown to hold if $\textrm{id}\in\ell^{2+\delta_0}(\Gamma)$, for some $\delta_0>2$,  it has been shown in  \cite{MenAsyIas} (theorem 4) that 
\begin{itemize}
\item if $2\alpha_0<\beta_0$ the chain is recurrent while if
$2\alpha_0>\beta_0$ the chain is transient;
\item if $2\alpha_0=\beta_0$ and $2\alpha_i-\beta_i-\beta_0=0$ for all $i: 0\leq i<k$ and there exists $i: 0\leq i<k$ such that $\beta_i>0$ then
\begin{itemize}
\item if $2\alpha_k-\beta_k-\beta_0\leq 0$ the chain is recurrent,
\item if $2\alpha_k-\beta_k-\beta_0>0$ the chain is transient.
\end{itemize}
\end{itemize} 

We assume that the  moment conditions  $\textrm{id}\in\ell^{2+\delta_0}(\Gamma)$ --- guaranteeing the condition~R for this model --- are satisfied through out this section.

Let $(\tilde{\xi}_n)$ be a Markov chain on $\BbX=\{0,1,\ldots\}$ satisfying for some $k\geq 0$
\[m(x)=\sum_{i=0}^{k-1} \frac{\beta_0} {2L_i(x)}
+\frac{\alpha_k} {L_k(x)} +o\left(\frac{1} {L_k(x)}\right)
\
\textrm{and}\
v(x)=\beta_0+o(1),\]
with $2a_k>\beta_0$. The aforementioned result  guarantees the transience of the chain; we term such a chain $k$-\textit{critically transient}.
When $2\alpha_k<\beta_0$ the previous result guarantees the recurrence of the chain; we term such a chain $k$-\textit{critically recurrent}.
 
In spite of its seemingly idiosyncratic character, this model proves universal as
Lyapunov functions used in the study of many general models in critical regimes  map those models to some $k$-critical  models on denumerably infinite unbounded subsets of $\BbR_+$.

\subsubsection{Moments of passage times}

\begin{prop}
\label{prop:k-critical-moments}
Let $(\xi_t)$ be a continuous-time Markov chain on $\BbX$ and $A$ be the finite set $A:=[0,x_0]\cap\BbX$ for some sufficiently large $x_0$. 
Suppose that for some integer $k\geq 0$, its embedded chain is $k$-critically 
recurrent, i.e.\
\[m(x)=\sum_{i=0}^{k-1} \frac{\beta_0} {2L_i(x)}
+\frac{\alpha_k} {L_k(x)} +o\left(\frac{1} {L_k(x)}\right)
\
\textrm{and}\
v(x)=\beta_0+o(1),\]
with $2a_k<\beta_0$. Denote $C=\alpha_k/\beta_0$ (hence $C<1/2$) and assume there exists  a constant $\kappa>0$ such that
\footnote{
i.e.\ there exists a constant $c_3>0$ such that $\frac{1}{c_3} \leq \frac{\gamma_x \ln_{(k)}^\kappa (x)}{L_k^2(x)}\leq  c_3$ for $x\geq x_0$.
}
 $\gamma_x=\cO(\frac{L_k^2(x)}{\ln_{(k)}^\kappa (x)})$ for large $x$.
 Define $p_0:=p_0(C,\kappa)= (1-2C)/\kappa$. 
\begin{enumerate}
\item  Assume that $v^{(\rho)}(x)=\cO(1)$ with $\rho=\max(2,1-2C)+\delta_0$. If $q<p_0$ then $\BbE_x\tau_A^q <\infty$.
\item Assume that $v^{(\rho)}(x)=\cO(1)$ with $\rho=\max(2,p_0)+\delta_0$. If $q\geq
p_0$ then $\BbE_x\tau_A^q =\infty$.
\end{enumerate}
\end{prop}

\begin{rema}{} 
We remark that when $\kappa\downarrow 0$ (for fixed $k$ and $C$) then $p_0\uparrow \infty$ implying that more and more moments exist.
\end{rema}

  \vskip3mm
\noindent\textit{Proof of the proposition \ref{prop:k-critical-moments}.}
For the function $f(x)=\ln_{(k)}^\eta x$ we determine
\[m_f(x)=\frac{1}{2}\beta_0\eta(2C+\eta-1+o(1))\frac{\ln_{(k)}^\eta(x)}{L_k^2(x)}.\]
\begin{enumerate}
\item 
For a $p>0$, we remark that
\[[0<p\eta< 1-2C \ \ \textrm{and}\ \  \eta\geq \frac{\kappa}{2}] 
 \Rightarrow 
\Gamma f^p(x) \leq -c f^{p-2}(x),\]
for some constant\footnote{The constant $c$ can be chosen $c\geq\frac{1}{2}\beta_0\eta (1-2C)c_3$.}  $c>0$.
Hence $p< 2(1-2C)/\kappa$ and statement 1 of the theorem \ref{thm:moments} implies
 for any $q\in]0,(1-2C)/\kappa[$, the $q$-th moment of the passage time exists. Optimising over the accessible values of $q$ we get  that $\BbE_x(\tau^q)<\infty$ for all $q<p_0$.
\item We distinguish two cases:
\begin{description}
\item{[$1-2C < \kappa$:]} We verify that, choosing $\eta\in]1-2C, \kappa]$ and 
$p>\frac{1-2C}{\kappa}$, the three conditions conditions of statement 2
of the theorem \ref{thm:moments}  (with $f=g$), namely
$\Gamma f \geq -c_1$,  $\Gamma f^r \leq c_2 f^{r-1}$, for some $r>1$, and $\Gamma f^p\geq 0$, outside some finite set $A$.
\item{[$ \kappa \leq 1-2C$:]} In this situation,  choosing the parameters  $\eta\in]0,\kappa]$ and $p>\frac{1-2C}{\kappa}$ implies simultaneous verification of the three conditions of statement 2 of the theorem \ref{thm:moments}.
\end{description}
In both situations, we conclude that  for all $q>\frac{1-2C}{\kappa}$, the corresponding moment does not exist. 
Optimising over the accessible values of $q$, we get  that $\BbE_x(\tau^q)=\infty$ for all $q>p_0$.
\end{enumerate}
\eproof

\subsubsection{Explosion and implosion}

\begin{prop}
\label{prop:k-crit-explo} Let $(\xi_t)$ be a continuous time Markov chain.
Suppose that for some integer $k\geq 0$, its embedded chain is $k$-critically transient.
 \begin{enumerate}
\item
If there exist a constant $d_1>0$, an integer $l>k$, and a  real $\kappa>0$ (arbitrarily small)  such that
\[\gamma_x\geq d_1L_k(x)L_l(x)  (\ln_{(l)}x)^{\kappa}, x\geq x_0,\]
then $\BbP_y(\zeta<\infty)=1$ for all $y\in\BbX$.
\item
If there exist a constant $d_2>0$ and an integer $l> k$ such that 
\[\gamma_x\leq d_2 L_k(x)L_l(x), x\geq x_0,\]
then the continuous-time chain is conservative.
\end{enumerate}
\end{prop}
\begin{proof}

\begin{enumerate}
\item
For a $k$-critically transient chain, chose a function $f$ behaving at large $x$ as $f(x)=\frac{1}{\ln_{(l)}^\eta(x)}$.  Denote $C=\alpha_k/\beta_0$ (hence $C>1/2$ for the chain to be transient).
We estimate then 
\[m_{f}(x)= -2\beta_0\eta (2C-1)\frac{1}{L_k(x)L_l(x)\ln_{(l)}^\eta x}+ o(\frac{1}{L_k(x)L_l(x)\ln_{(l)}^\eta x}).\]
We conclude by theorem \ref{thm:explo}.
\item Choosing as Lyapunov function the identity function $f(x)=x$ and  estimate
\[m_{\ln_{(l+1)}\circ f}(x)= 2\beta_0(2C-1)\frac{\ln_{(l+1)}x}{L_k(x)L_{l+1}(x)}=
2\beta_0(2C-1)\frac{1}{L_k(x)L_{l}(x)}.\]
We conclude by  theorem \ref{thm:non-explo} by choosing $g(s) =L_l(s)$.
\end{enumerate}
\end{proof}

\begin{prop}
\label{prop:k-crit-implo}
Let $(\xi_t)$ be a continuous time Markov chain.
Suppose that for some integer $k\geq 0$, its embedded chain is $k$-critically 
recurrent. Denote $C=\alpha_k/\beta_0$ (hence $C<1/2$ for the chain to be recurrent). Let  $A$ be the finite set $A:=[0,x_0]\cap\BbX$ for some sufficiently large $x_0$.
\begin{enumerate}
\item
If there exist a constant $d_1>$, an integer $l>k$, and an arbitrarily small  real $\kappa>0$  such that
\[\gamma_x\geq d_1 L_k(x)L_l(x)  (\ln_{(l)}x)^{\kappa}, x\geq x_0,\]
then there exists a constant $B$ such that  $\BbE_y\tau_A \leq B$, uniformly in $y\in A^c$, i.e.\ the chain implodes.
\item
If there exist a constant $d_2>0$ and an integer $l> k$ such that 
\[\gamma_x\leq d_2 L_k(x)L_l(x), x\geq x_0,\]
then the continuous time chain does not implode.
\end{enumerate}
\end{prop}
\begin{proof}
\begin{enumerate}
\item
Use the function $f$ defined for sufficiently large $x$ by the formula
$f(x)=1-\frac{1}{\ln_{(l)}^\eta x}$, for some $l> k$ and $\eta>0$.
We estimate then 
\[m_f(x) = 2\beta_0 \eta (2C-1)  \frac{1}{L_k(x)L_l(x)\ln_{(l)}^\eta x}+ o(\frac{1}{L_k(x)L_l(x)\ln_{(l)}^\eta x}).\]
We conclude by statement 1 of theorem \ref{thm:implo}.
\item
Using the function $f$ defined for sufficiently large $x$ by the formula
$f(x)=\ln_{(l+1)}^\eta x$, for some $l\geq k$.
We estimate then 
\[m_f(x) = 2\beta_0 \eta (2C-1)  
\frac{\ln_{(l+1)}^\eta x}{L_k(x)L_{l+1}(x)}+ o(\frac{\ln_{(l+1)}^\eta x}
{L_k(x)L_{l+1}(x)}).\]
If $\gamma_x \leq d_2 L_k(x) L_l(x)$ for large $x$, then,
using the above estimate for the case $\eta=1$ and the case $\eta=r$ for some small $r>1$,
we observe that the conditions $\Gamma f  \geq -\epsilon$ and
$\Gamma f^r \leq f^{r-1}$ are simultaneously verified. We conclude by the statement 2 of the theorem \ref{thm:implo}.
\end{enumerate}
\end{proof}

%%%%% RANDOM WALK %%%%%%%%%%%%%%
 
\subsection{Simple random walk on $\BbZ^d$ for $d=2$ and $d\geq 3$}
\label{ssec:srw}
Here the state space $\BbX=\BbZ^d$ and the embedded chain is a simple random walk on $\BbX$. Since in dimension 2 the simple random walk is null recurrent while in dimension $d\geq 3$ is transient, a different treatment is imposed.

\subsubsection{Dimension $d\geq 3$}
For the Lyapunov function $f$ defined by 
$\BbZ^d\ni x \mapsto f(x):=\|x\|$, we can show that there exist constants 
$\alpha_0>0$ and $\beta_0>0$ such that $\lim_{\|x\|\rightarrow \infty} \|x\| m_f(x)=\alpha_0$ and
 $\lim_{\|x\|\rightarrow \infty}  v_f(x)=\beta_0$ such that $C=\alpha_0/\beta_0> 1/2$. Therefore the one dimensional process $X_t=f(\xi_t)$ has $0$-critically transient Lamperti behaviour.
 
We get therefore that  $(\xi_t)_{t\in\BbR_+}$ is a 
(quite unsurprisingly) transient process and that if there exist a constant $a>0$ and 
\begin{itemize}
\item a constant $d_1>0$, an integer $l>0$, and a  real $\kappa>0$ (arbitrarily small)  such that
\[\gamma_x\geq d_1\|x\|L_l(\|x\|)  (\ln_{(l)}\|x\|)^{\kappa}, \|x\|\geq a,\]
then $\BbP_y(\zeta<\infty)=1$ for all $y\in\BbX$;
\item
a constant $d_2>0$ and an integer $l> 0$ such that 
\[\gamma_x\leq d_2 \|x\|L_l(\|x\|), \|x\|\geq a,\]
then the continuous time chain is conservative.
\end{itemize}

\subsubsection{Dimension 2}

Using again the Lyapunov function $f(x)=\|x\|$, we show that the one dimensional process $X_t=f(\xi_t)$ is of the  $1$-critically recurrent Lamperti type. Hence, again using the results obtained in \S \ref{ssec:lamperti}, 
we get that if there exist a constant $a>0$ and 
\begin{itemize}
\item
a constant $d_1>0$, an integer $l>1$, and an arbitrarily small  real $\kappa>0$  such that
\[\gamma_x\geq d_1 L_1(\|x\|)L_l(\|x\|)  (\ln_{(l)}\|x\|)^{\kappa}, \|x\|\geq a,\]
then there exists a constant $C$ such that  $\BbE_y\tau_A \leq C$, uniformly in $y$ for $y: \|y\|\geq a $, i.e.\ the chain implodes;
\item
a constant $d_2>0$ and an integer $l> 1$ such that 
\[\gamma_x\leq d_2 L_1(\|x\|)L_l(\|x\|), \|x\|\geq a,\]
then the continuous time chain does not implode.
\end{itemize}

%%%%%%%% WEDGE %%%%%%%%%%%%%%%%

\subsection{Random walk on $\BbZ_+^2$ with reflecting boundaries}
\label{ssec:quadrant}
\subsubsection{The model in discrete time}
Here $\BbX=\BbZ_+^2$. We denote by $\oc{\BbX}=\{x\in\BbZ_+^2:
x_1>0, x_2>0\}$ the \textit{interior} of the wedge 
and by $\partial_1 \BbX=\{x\in\BbZ_+^2:
x_2=0\}$ (and similarly for $\partial_2 \BbX$) its \textit{boundaries}. Since $\BbX$ is a subset of a vector space, we can
define  directly the increment vector $D:= \tilde{\xi}_{n+1}-\tilde{\xi}_{n}$ and the
 average conditional  drift $m(x):=m_{\textrm{id}}(x)=\BbE(D|\tilde{\xi}_{n}=x)\in\BbR^2$. We assume that
for  all $x\in\oc{\BbX}$, $m(x)=0$ so that we are in a critical regime. For $x\in\partial_\flat \BbX$; with $\flat=1,2$, the drift $m^\flat(x)$ does not vanish but is a constant vector $m^\flat$ that forms angles $\phi^\flat$ with respect to the normal to $\partial_\flat\BbX$. For $x\in\oc{\BbX}$, the conditional covariance matrix $C(x):=(C(x)_{ij})$,
with
$C(x)_{ij}(=  \BbE[D_i D_j|\tilde{\xi}_{n}=x]$, is the constant $2\times2$ matrix $C$, reading 
\[C:=\Cov(D,D)=\begin{pmatrix} s_1^2 & \lambda\\ \lambda & s_2^2 \end{pmatrix}.\]
There exists an isomorphism $\Phi$ on $\BbR^2$  such that $\Cov(\Phi D, \Phi D) = \Phi C \Phi^t = I$; it is elementary to show that 
\[\Phi=\begin{pmatrix} \frac{s_2}{d} & -\frac{\lambda}{s_2 d}\\ 0 & \frac{1}{s_2} \end {pmatrix},\] where $d=\sqrt{\det C}$, is a solution to the aforementioned isomorphism equation. This isomorphism maps the quadrant $\BbR^2_+$ into a squeezed wedge $\Phi(\BbR^2_+)$ having an angle $\psi$ at its summit reading $\psi=\arctan (-d^2/\lambda)$. Obviously $\psi=\pi/2$ if $\lambda=0$, while $\psi\in ]0, \pi/2[$ if $\lambda<0$ and
$\psi\in ]\pi/2, \pi[$ if $\lambda >0$. We denote $\BbY=\Phi(\BbX)$ the squeezed image of the lattice. The isomorphism $\Phi$ transforms the average drifts at the boundaries into  $n^\flat =\Phi m^\flat$ forming new  angles, $\psi_\flat$, with the normal to the boundaries of $\Phi(\BbR_+^2)$.

The discrete time model has been exhaustively treated in \cite{AspIasMen} and its extension to the case of excitable boundaries carrying internal states
in \cite{MenPet-wedge}. Here we recall the main results of \cite{AspIasMen}
under some simplifying assumptions that allow us to present them here without redefining completely the model or considering all the technicalities. 
The assumptions we need are that the jumps
\begin{itemize}
\item are bounded from below, i.e.\
there exists a constant $K>0$ such that $D_1\geq -K$ and $D_2\geq -K$,
\item satisfy a sufficient integrability condition, for instance
$\BbE(\|D\|_{2+\delta_0}^ {2+\delta_0} )<\infty$ for some $\delta_0>0$, 
\item are such that their covariance matrix is non degenerate.
\end{itemize}
Under these assumptions 
we can state the following simplified version of the results in \cite{AspIasMen}.

Denote by $\chi=(\psi_1+\psi_2)/\psi$ and $A=\{x\in\BbX: \|x\|\leq a\}$. 
\begin{enumerate}
\item If $\chi\geq 0$ the chain is recurrent.
\item If $\chi< 0$ the chain is transient.
\item If $0<\chi<2+\delta_0$, then for every $p<\chi/2$ and every $x\not\in A$,
$\BbE_x \tau_A^p <\infty$.
\item If $0<\chi<2+\delta_0$, then for every $p>\chi/2$ and every $x\not\in A$,
$\BbE_x \tau_A^p =\infty$.
\end{enumerate}

Let $f:\BbR^2_+\rightarrow \BbR_+$ be a $C^2$ function. Define $\tilde{f}(y)=f(\Phi^{-1} y)$. Although the Hessian operator does not in general transform as a tensor, the linearity of $\Phi$ allows
however to write 
$\Hess_f(x) = \Phi^t \Hess_{\tilde{f}}(\Phi x) \Phi$. 
For every $x\in\oc{\BbX}$ we establish then\footnote{Since we have used the symbol $\Delta$ to denote the jumps of the process, we introduce the symbol $\Lap$ to denote the Laplacian.} the identity:
\cut{
\begin{eqnarray*}
\BbE(\bra D, \Hess_f(x) D\ket| \tilde{\xi}_n=x)&=&
\sum_{i,j} \Cov ((\Phi D)_i, (\Phi D)_j)|  \tilde{\xi}_n=x) (\Hess_{\tilde{f}}(\Phi x))_{ij}\\
&=& \tr \Hess_{\tilde{f}}(\Phi x)=  \Lap_{f\circ \Phi^{-1}}\circ\Phi(x).
\end{eqnarray*}
}
\[\BbE(\bra D, \Hess_{f\circ \Phi} (x) D\ket| \tilde{\xi}_n=x)= \Lap_{f\circ\Phi}(x).\]

We denote by $h_{\beta, \beta_1}(x)=\|x\|^\beta \cos(\beta \arctan (\frac{x_2}{x_1})-\beta_1)$. Then this function is harmonic, i.e.\ $\Lap_{h_{\beta, \beta_1}}=0$.
We are interested in harmonic functions that are positive on $\Phi(\BbR^2_+)$; positivity and geometry impose then conditions on $\beta$ and $\beta_1$.
In fact, $\sign(\beta)\beta_1$ is the angle of 
$\nabla h_{\beta, \beta_1}(x)$ at $x\in\partial_1\BbX$,  with the normal to 
$\partial_1\BbX$. Similarly, if $\beta_2=\beta \psi-\beta_1$, then $\sign(\beta)\beta_2$ is the angle of the gradient with the normal to $\partial_2\BbX$. Now, it becomes evident that
$\beta_i$, $i=1,2$, must lie in the interval $]-\pi/2,\pi/2[$ and subsequently $\beta=\frac{\beta_1+\beta_2}{\psi}$. Notice also that the datum of two admissible angles  $\beta_1$ and $\beta_2$ uniquely determines the harmonic function whose gradient at the boundaries forms angles as above. 
Hence, $\bra \nabla h_{\beta,\beta_1} (y), n^\flat\ket=
\|y\|^\beta\beta\sin(\psi_\flat-\beta_\flat)$, for $y\in \partial_\flat\BbY$ and 
$\flat=1,2$. 

Let now $g:\BbR_+\rightarrow \BbR_+$ be a $C^3$ function and $h=h_{\beta,\beta_1}$ an harmonic function that remains positive in $\Phi(\BbR^2_+)$.  On denoting $y=\Phi x$,  abbreviating $ \Xi:=
g(h(\Phi \tilde{\xi}_{n+1}))-g(h(\Phi \tilde{\xi}_{n}))$, 
and using the fact that $h$ is harmonic, we get
\begin{eqnarray*}
\BbE(\Xi|\tilde{\xi}_{n}=x)&=&
g'(h(y)) \BbE(\bra \nabla h(y), \Phi D\ket|  \tilde{\xi}_{n}=x) \\
&&+\frac{g''(h(y))}{2}
 \BbE(\bra \nabla h(y), \Phi D\ket ^2|  \tilde{\xi}_{n}=x)\\
 &&+
 \frac{[g'(h(y))]^2}{2}
\BbE(\bra \Phi D, \Hess_{h} (y) \Phi D\ket| \tilde{\xi}_n=x)+ R_3\\
&=& 
g'(h(y)) \bra \nabla h(y), n(y)\ket + \frac{g''(h(y))}{2} \|\nabla h(y)\|^2+R_3(y),
\end{eqnarray*}
where $R_3$ is the remainder of the Taylor expansion.
The value of the  conditional increment depends on the position of $x$. If $x\in\partial_\flat \BbX$ the dominant term of the right hand side is 
$g'(h(y)) \bra \nabla h(y), n^\flat\ket$, while in the interior of the space, that term strictly vanishes because there $n(y)=0$; hence the dominant term becomes the term
$\frac{g''(h(y))}{2} \|\nabla h(y)\|^2$. 

\subsubsection{The model in continuous-time}

\begin{prop}
\label{prop-quad-moments}
Let $0<\chi=(\psi_1+\psi_2)/\psi$ (hence the chain is recurrent) and $A:=A_a=\{x\in\BbX: \|x\|\leq a\}$ for some $a>0$, and $\gamma_x=\cO(\|x\|^{2-\kappa})$; denote $p_0=\chi/\kappa$. Suppose further that $\textrm{id}\in\ell^{\rho}(\Gamma)$ for some $\rho>2$.
\begin{enumerate}
\item If $q<p_0$, then $\BbE_x(\tau_A^q)<\infty$.
\item If $q>p_0$, then $\BbE_x(\tau_A^q)=\infty$.
\end{enumerate}
\end{prop}
\begin{proof} 
Consider the Lyapunov function $f(x) = h_{\beta,\beta_1}(x)^\eta$.
\begin{enumerate}
\item If $0<p\eta<1$ then $m_{f^p}(x)<0$. The condition 
$\Gamma f^p \leq -cf^{p-2}$ reads then 
$\gamma_x \geq C \frac{h_\beta^{p\eta-2\eta} (x)}
{h_\beta^{p\eta-2}\|\nabla h_\beta\|^2}=
C'\frac{\|x\|^{2\beta-2\beta\eta}}{\|x\|^{2\beta-2}}=C' \|x\|^{2-2\beta\eta}$
from which follows that $2\beta\eta>\kappa$. This inequality, combined with
$0<p\eta<1$, yields that  for all $q<\frac{\beta}{\kappa}<\frac{\chi}{\kappa}$, $\BbE_x(\tau_A^q)<\infty$.
Hence, on optimising on the accessible values of $q$ we obtain the value of $p_0$.
\item
We proceed similarly; we need however to use the full-fledged version of the statement 2 of theorem \ref{thm:moments}, with both the function $f$ and $g=h_{\chi,\psi_1}^\eta$. Then  $f(x)\leq C g(x)$ and 
\[[\eta\beta <\kappa \ \textrm{and}\ \ \eta p>1]
\Rightarrow [ \Gamma g\geq -\epsilon \ \textrm{and} \ \Gamma g^r \leq cg^{r-1} \textrm{for}\ r>1, \ \textrm{and} \ \Gamma f^p \geq 0].\]
Simultaneous verification of these inequalities yields $p_0=\chi/\kappa$.
\end{enumerate}
\end{proof}

Using again Lyapunov functions of the form $f=h_\beta^\eta$ we can show that the drift of the chain in the transient case can be controlled by two $0$-critically transient Lamperti processes in the variable $\|x\|$ that are uniformly comparable. We can thus show, using methods developed in \S \ref{ssec:lamperti} the following 

\begin{prop}
\label{prop-quad-explo}
Let  $\chi< 0$ (hence the chain is transient).
\begin{enumerate}
\item If there exist a constant $d_1>0$  and an arbitary integer $l>0$ such that 
$\gamma_x \geq d_1\|x\|L_l(\|x\|) \ln_{(l)}^{\kappa}\|x\|$, for some arbitrarily small $\kappa>0$, then the chain explodes.
\item If there exist a constant $d_2>0$ and an arbitary integer $l>0$ such that 
$\gamma_x \leq d_2\|x\|L_l(\|x\|)$, then the chain does not explode.
\end{enumerate}
\end{prop}

With the help of similar arguments we can show the following 
\begin{prop}
\label{prop-quad-implo}
Let  $0<\chi$ (hence the chain is recurrent).
\begin{enumerate}
\item If there exist a constant $d_1>0$  and an arbitrary integer $l>0$ such that 
$\gamma_x \geq d_1\|x\|L_l(\|x\|) \ln_{(l)}^{\kappa}\|x\|$, for some arbitrarily small $\kappa>0$, then the chain implodes.
\item If there exist a constant $d_2>0$ and an arbitrary integer $l>0$ such that 
$\gamma_x \leq d_2\|x\|L_l(\|x\|)$, then the chain does not implode.
\end{enumerate}
\end{prop}

\subsection{Collection of one-dimensional complexes}
\label{ssec:fingers}
We introduce some  simple models to illustrate two phenomena:
\begin{itemize}
\item it is possible to have $0<\BbP_x(\zeta<\infty)<1$,
\item it is possible to have $\BbP_x(\zeta=\infty)=0$ and $\BbE_x\zeta=\infty$.
\end{itemize}

The simplest situation corresponds to a continuous-time Markov chain whose embedded chain is a simple  transient random walk on $\BbX=\BbZ$ with non trivial tail boundary. For instance, choose some $p\in]1/2,1[$ and  transition matrix
\[P_{xy}= \left\{\begin{array}{ll} 
1/2 & \textrm{if}\ x=0, y=x\pm1;\\
p & \textrm{if}\ x>0, y=x+1\ \textrm{or}\ x<0, y=x-1;\\
1-p & \textrm{if}\ x>0, y=x-1\ \textrm{or}\ x<0, y=x+1;\\
0 & \textrm{otherwise.}
\end{array}\right.\]
Then, for every $x\ne 0$, $0<\BbP_x(\tau_0 = \infty)<1$.
Suppose now that $\gamma_x=c$ for $x<0$ while there exists a sufficiently large integer $l\geq 0$ and an arbitrarily small $\delta >0$ such that
$\gamma_x \geq c L_l(x) \ln_{(l)}^\delta x$ for $x\geq x_0$. Then, using theorem \ref{thm:cond-explo}, we establish that
$\BbE_x(\zeta| \tau_{\BbZ_-}>\zeta) <\infty$ for all $x>0$ while 
 $\BbP_x(\zeta=\infty| \tau_{\BbZ_+}=\infty) =1$ for all $x<0$. This result combined with irreducibility of the chain leads to the conclusion:
 $0<\BbP_x(\zeta<\infty)<1$ for all $x\in\BbX.$ 

It is worth noting that bending the axis $\BbZ$ at 0 allows considering the state space  as the gluing of two one-dimensional complexes $\BbX_2=\{0\}\cup\BbN\times\{-,+\}$; every point $x\in\BbZ\setminus\{0\}$ is now represented as $x=(|x|, \sgn(x))$. This construction can be generalised by  gluing  a denumerably infinite family of one-dimensional complexes through a particular point $o$ and introducing the state space $\BbX_\infty=\{o\}\cup\BbN\times\BbN$; every point $x\in\BbX_\infty\setminus\{o\}$ can be represented as 
$x=(x_1,x_2)$ with $x_1,x_2\in\BbN$.  

Let $(\xi_t)_{t\in[0,\infty[}$ be a continuous time Markov chain evolving on the state space $\BbX_\infty$. Its embedded (at the moments of jumps) chain $(\tilde{\xi}_n)_{n\in\BbN}$ has 
transition matrix given by
\[P_{xy}= \left\{\begin{array}{ll} 
\pi_{y_2} & \textrm{if}\ x=o, y=(1,y_2), y_2\in\BbN,\\
p & \textrm{if}\ x=(x_1,x_2), y=(x_1+1, x_2), x_1\geq 1, x_2\in\BbN,\\
1-p & \textrm{if}\  x=(x_1,x_2), y=(x_1-1, x_2), x_1>1, x_2\in\BbN,\\
1-p & \textrm{if}\  x=(1,x_2), n\in\BbN, y=o,\\
0 & \textrm{otherwise,}
\end{array}\right.\]
where $1/2<p<1$ and $\pi=(\pi_n)_{n\in\BbN}$ is a probability vector on $\BbN$, satisfying $\pi_n>0$ for all $n$. The chain is obviously irreducible and transient. 

The space $\BbX_\infty$ must be thought as a ``mock-tree'' since, for transient Markov chains,  it has a sufficiently rich boundary structure  without any of the complications of the homogeneous tree (the study on full-fledged trees is postponed in a subsequent publication). Suppose that for every $n\in\BbN$ there exist an integer $l_n\geq 0$,  a real $\delta_n>0$, and a $K_n>0$ such that 
for $x=(x_1,x_2)\in\BbN\times\BbN$ for $x_1$ large enough, $\gamma_{x}$ satisfies  $\gamma_{(x_1,x_2)}=K_{x_2}\cO( L_{l_{x_2}}(x_1) \ln^{\delta_{x_2}}_{(l_{x_2})} x_1)$.

By applying theorem  \ref{thm:cond-explo}, we establish that
$\mathcal{Z}_{x_2}:=\BbE_{(x_1,x_2)}(\zeta| \tau_{o}=\infty) <\infty$ for all $x_1>0$ and all $x_2\in\BbN$, hence $\BbP_{(x_1,x_2)}(\zeta=\infty| \tau_{o}=\infty) =0$. Irreducibility 
implies then that $\BbP_{o}(\zeta=\infty) =0$.
However, 
\[\BbE_o (\zeta) \geq \sum_{x_2\in\BbN} \pi_{x_2} \BbE_{(1,x_2)}(\zeta|  \tau_{o}=\infty)\BbP_{(1,x_2)}(\tau_{o}=\infty)=
\frac{2p-1}{p} \sum_{x_2\in\BbN} \pi_{x_2} \mathcal{Z}_{x_2}.\] 
Since the  sequences 
$(l_n)_n$, $(\delta_n)_n$, and $(K_n)_n$ are totally arbitrary, while the
positive sequence $(\pi_n)_n$ must solely satisfy the probability constraint
$\sum_{n\in\BbN} \pi_n=1$,  all possible behaviour for $\BbE_o\zeta$ can occur. In particular, we can choose, for all $n\in\BbN$,  $l_n=0$ and $\delta_n=1$; this choice  gives  $\gamma_{(x_1, n)} = K_n\cO(\frac{1}{x_1^2})$ for every $n$ and for large $x_1$, leading
to the estimate $\mathcal{Z}_n\geq CK_n$, for all $n$. 

Choosing now, for instance, $\pi_n=\cO(1/n^2)$ and $K_n=\cO(n)$ for large $n$, we get 
\[\BbP_{o}(\zeta=\infty) =0\  \textrm{and}\ \BbE_o\zeta=\infty.\] 
 
This remark leads naturally to the  question whether for transient exploding chains with non-trivial tail boundary, there exists some critical $q>0$ such that
$\BbE (\zeta^p)<\infty$ for $p<q$ while $\BbE (\zeta^p)=\infty$ for $p>q$. Such models include continuous time random walks on the homogeneous tree and more generally on non-amenable groups.
These questions are currently under investigation and are postponed to a subsequent publication.

\bibliographystyle{plain}
%\frenchspacing
\scriptsize{
\bibliography{rwre,matrix,petritis}
}

\end{document}